\documentclass[12pt,reqno]{amsart}
\usepackage{amsmath}
\usepackage{amssymb}
\usepackage{amstext}
\usepackage{a4wide}
\usepackage{graphicx}
\allowdisplaybreaks \numberwithin{equation}{section}
\usepackage{color}
\usepackage{cases}

\numberwithin{equation}{section}

\newtheorem{theorem}{Theorem}[section]
\newtheorem{proposition}[theorem]{Proposition}
\newtheorem{corollary}[theorem]{Corollary}
\newtheorem{lemma}[theorem]{Lemma}

\theoremstyle{definition}

\theoremstyle{remark}
\newtheorem{remark}[theorem]{Remark}

\begin{document}

\title
{Multi-soliton dynamics in the nonlinear Schr\"{o}dinger equation}

 \author{Daomin Cao, Qing Guo, Changjun Zou}
\address{Institute of Applied Mathematics, Chinese Academy of Sciences, Beijing 100190, and University of Chinese Academy of Sciences, Beijing 100049,  P.R. China}
\email{dmcao@amt.ac.cn}

\address{College of Science, Minzu University of China, Beijing 100081,  P.R. China}
\email{guoqing0117@163.com}

\address{Institute of Applied Mathematics, Chinese Academy of Sciences, Beijing 100190, and University of Chinese Academy of Sciences, Beijing 100049,  P.R. China}
\email{zouchangjun17@mails.ucas.ac.cn}


\begin{abstract}
In this paper, we study the Cauchy problem of the nonlinear Schr\"{o}dinger equation with a nontrival potential $V_\varepsilon(x)$. In particular, we consider the case where the initial data is close to a superposition of $k$ solitons with prescribed phase and location, and investigate the evolution of the Schr\"{o}dinger system. We prove that over a large time interval with the maximum time tending to infinity, all $k$ solitons will maintain the shape, and the solitons dynamics can be regarded as an approximation of $k$ particles moving in $\mathbb{R}^N$ with their accelerations dominated by $\nabla V_\varepsilon$, provided the barycenters of these solitons do not coincide.  
\end{abstract}

\maketitle

\section{Introduction}
We will consider the following Cauchy problem of the nonlinear Schr\"{o}dinger equation (NSE) in $\mathbb{R}^1\times\mathbb{R}^{N}$:
\begin{numcases}
{}
\label{1-1} \mathbf{i}\varepsilon\partial_t\psi=-\varepsilon^2\Delta\psi+V_\varepsilon(x)\psi-f(\psi), \\
\label{1-2} \psi(0,x)=\psi_0(x), 
\end{numcases}
where $\mathbf{i}=\sqrt{-1}$ is the imaginary unit, $\varepsilon$ denotes the Planck constant, $\psi: \mathbb{R}^1\times\mathbb{R}^{N}\to \mathbb{C}$, $V_\varepsilon(x)$ is a real-valued potential function, $f(\psi)$ is the nonlinear term mapping the complex Sobolev space $H^1(\mathbb{R}^N,\mathbb{C})$ to $H^{-1}(\mathbb{R}^N,\mathbb{C})$ satisfiying $f(0)=0$ and $f(e^{\mathbf{i}\theta}u)=e^{\mathbf{i}\theta}f(u)$. Over last decades, the NSE was investigated by many mathematians since it arises in the theory of Bose-Einstein condensation, nonlinear optics, theory of water waves and in many other areas.

For simplicity, we will first study a version of \eqref{1-1} and \eqref{1-2}:
\begin{equation}\label{1-3}
\begin{cases}
\mathbf{i}\partial_t\psi=-\varepsilon^2\Delta\psi+V_\varepsilon(x)\psi-f(\psi),\\
\psi(0,x)=\psi_0(x),\\
\end{cases}
\end{equation}
where we do not scale the time parameter. Then, we can scale on time to revisit the property of solutons to Cauchy problem \eqref{1-1} and \eqref{1-2}. The suitable $H^{1}(\mathbb{R}^N,\mathbb{C})$ norm for this problem is given by 
\begin{equation}\label{1-4}
||u||_{H^1}^2:=\int_{\mathbb{R}^N}\frac{1}{\varepsilon^N}(\varepsilon^2|\nabla u|^2+|u|^2)dx,
\end{equation}
with $|z|^2=(\mathbf{Re}z)^2+(\mathbf{Im}z)^2$ for a complex number $z$. By classical regularity theory of the NSE, when $f$ is of the special form
\begin{equation}\label{1-5}
f(\psi)=\lambda|\psi|^{p-1}\psi
\end{equation}
with $1<p<1+\frac{4}{N}$ and $\lambda>0$, the nonlinear Schr\"{o}dinger equation has a global solution $\psi(x,t)\in C(\mathbb{R}^N; H^1(\mathbb{R}^N,\mathbb{C}))\cap C(\mathbb{R}^N; H^{-1}(\mathbb{R}^N,\mathbb{C}))$, provided the initial data $\psi(0,x)=\psi_0(x)$ is in $H^1(\mathbb{R}^N,\mathbb{C})$. Readers can consult \cite{Caz,Ka} for more results on the Cauchy problem of the nonlinear Schr\"{o}dinger equation. In the next section, we will impose proper constraints on the potential $V_\varepsilon$ and the nonlinearity $f$.

When $V_\varepsilon(x)\equiv 0$, if $f$ is given by \eqref{1-5} and $\lambda$ is large enough, \eqref{1-3} can have a stable solitary wave solution of the form
\begin{equation}\label{1-6}
\psi(x,t)=e^{\mathbf{i}(\frac{1}{2}v\cdot(\frac{x-a}{\varepsilon})+\gamma)}\eta_{\mu,\varepsilon}(x-a),
\end{equation}
where $v\in \mathbb{R}^N$ denotes the velocity of solitary wave, $a=vt+a_0$ with the initial location $a_0\in \mathbb{R}^N$ , $\gamma=\mu t+\frac{|v|^2}{4}t+\gamma_0$ with the initial phase $\gamma_0$, $\eta_{\mu,\varepsilon}$ is a radially symmetric solution of following nonlinear eigenvalue problem
\begin{equation}\label{1-7}
-\varepsilon^2\Delta u+\mu u-f(u)=0, \ \ \ x\in \mathbb{R}^N
\end{equation}
with $\mu>0$. In this case, solitary waves travel with a constant velocity $v$ and an oscillatory phase $\mu t-\frac{1}{4}|v|^2t+\gamma_0$. From \cite{Gi} we know that $\eta_{\mu,\varepsilon}>0$ satisfies $\lim\limits_{|x|\to\infty}\nabla\eta_{\mu,\varepsilon}/\eta_{\mu,\varepsilon}=-1$ and
\begin{equation}\label{1-8}
\lim\limits_{|x|\to\infty}\eta_{\mu,\varepsilon}(x)e^{\alpha\frac{|x|}{\varepsilon}}=0
\end{equation} 
with some $\alpha>0$. Thus as $\varepsilon\to 0^+$, the mass of the solitary wave, which is denoted by the $L^2$ norm, will concentrate and behave like a particle. 

Solitary waves with this concentrating property are usually called solitons, and the process $\varepsilon\to 0^+$ is known as semiclassical limit of the nonlinear Schr\"{o}dinger equation. The existence of standing wave solutions with solitons located near nondegenerate and isolated critical points of potential function was proved by a finite dimentional reduction method. For equations with a nonsymmetric nonlinear term $Q(x)|\psi|^{p-1}\psi$, multi-peak solutions are constructed in \cite{CY1,CY2}. However, in this case, the location of solitons is determined by $Q(x)$ instead of the potential function. For relevant result we refer to \cite{CY3,WY}, in which different kinds of steady soliton solutions were constructed.

Now, we turn to soliton evolution in general case where $V_\varepsilon(x)\not\equiv 0$. Inspired by above observation, for the case where initial data is close to a single soliton with given velocity and phase, it is guessed that the soliton will maintain its shape, and move under the action of potential $V_\varepsilon(x)$. In \cite{Br}, for nonlinearity given by \eqref{1-5}, Bronski and Jerrard proved that if the initial data satisfies foresaid condition (be a small perturbation of $\psi$ given by \eqref{1-6}), as $\varepsilon\to 0^+$, the solution $\psi(x,t)$ of \eqref{1-1} and \eqref{1-2} satisfies
\begin{equation}\label{1-9}
\frac{1}{\varepsilon^N}|\psi(\frac{x}{\varepsilon},\frac{t}{\varepsilon})|dx\to ||\eta_{\mu,\varepsilon}||^2_{L^2(\mathbb{R}^N)}\delta_{a(t)}
\end{equation}
in the $C^{1*}(\mathbb{R}^N)$, which is the dual to $C^1(\mathbb{R}^N)$. They have showed that $a(t)$, the barycenter of soliton, satisfies $\frac{1}{2} \ddot a=\nabla V_\varepsilon(x)+o(\varepsilon)$. Then Keraani \cite{Ke1,Ke2} devoloped their method and gave a more precise estimate for the dynamics. In \cite{B1,B2,B3}, Benci, Ghimenti and Micheletti also work on this problem. By analysis of total energy and a classical concentration-compactness lemma \cite{L1,L2,St,W}, they obtained similar dynamics for single soliton. However, in order to ensure the global existance of stable solitons, they gave a blow up to the nonlinearity term so that the energy is dominated by $f(\psi)$.

The foresaid results were obtained on the basis of orbital stability of NLS solitary waves, which was first proven by Cazenave and Lions for trival potential in \cite{Caz1}. For nontrival $C^2$ potential with a nondegenerate minimum point $x_0$, Grillakis et al. \cite{Gri1,Gri2} proved the orbital stability of standing solitary wave whose center is located near $x_0$. To obtain solitons dynamics in an external potential, a further uniqueness result is needed, which claims a function $\psi\in H^1(\mathbb{R}^N)$ is close to a ground state solution in $H^1$ sence, if and only if its energy is near the minimal energy. Weinstein \cite{We1,We2} proved this uniqueness of ground state solutions in 1 and 3 space dimensions. In \cite{Kw}, Kwong combined with Weinstein's arguments and proved the same result in all space dimensions.

It is worth mentioning that all above results are on the dynamics of a single soliton with nontrival potential so far. We may ask the following questions: If we put $k$ solitons in $\mathbb{R}^N$ and let them envolve under the domination of \eqref{1-1} and \eqref{1-2}, then how will their trajectories be like? And will these solitons influence each other? In this paper, we will exactly focus on this topic and give the dynamics formula of multi-peak evolution under appropriate assumptions. In particular, we will prove that if the barycenters of these $k$ solitons do not coincide, then for $t\in (0,\frac{T_0}{\varepsilon}]$ the solitons dynamics are mainly determined by $\nabla V_\varepsilon$, while the interaction among the solitons can be ignored. 

To study the multi-solitons dynamics, traditional method is invalid since we can not use a minimization argument on energy, which is crucial in \cite{B1,B2,B3,Br}. We will put forward a new method, which is inspired by Fr\"{o}hlich et al. \cite{Fr1,Fr2} investigating the Hartree equations, and \cite{Fr3,Fr4} on the dynamics of solitary waves. The main idea of this paper is introducing the Plank constant $\varepsilon$ as a scaling parameter, and decomposing the solution $\psi(x,t)$ as $\Phi(x,t)+\omega(x,t)$, where $\Phi(x,t)$ is in the multi-soliton manifold $\mathbf{M}_k$ (defined in Section 4 and 5), and $\omega(x,t)$ is the perturbation term. In order to obtain the dynamic for each single soliton, we will make smooth truncations on $\psi$ to separate each soliton from others and regard the truncated function as an approximation of solitons. These truncations will cause great difficulty since the conservation of mass and momentum doesn't hold any more, and there is interaction part from other $k-1$ solitons in every truncated function. Thanks to the transport formula for mass and momentum in NSE (introduced in Section 3) and the exponential decay rate of ground state solutions \eqref{1-8}, we can give a precise estimate for each truncated part of $\psi$ and prove that the interaction among $k$ solitons is indeed a small term. 

In \cite{Ab}, Abou Salem et al. also considered this problem using a similar decomposition and studied the case where solitons get close to each other with barycenters not coinciding. However, they mainly focused on the conservation of mass and estimate the $L^2$ norm of perturbation term $w$. While we give a more precise result by dividing $w$ into $k$ parts and estimating the $H^1$ norm of each part using the energy method. As a consequence of this difference, in \cite{Ab} the potential $V_\varepsilon$ can depend on $t$, but the local nonlinearity can only be of the form $f=|\psi|\cdot\psi$ when $|\psi|$ is very large. While we improve the local nonlinearity to the form $f=|\psi|^{p-1}\psi$ with $1<p<1+\frac{4}{n}$, at the cost of requiring $V_\varepsilon$ only depends on $x$. On the other hand, the maximum time in \cite{Ab} depends on large initial relative velocity $v_{i,j}(0)$ or large initial distance $d_{i,j}(0)$ of two different solitons. While in our paper, we do not assume $|v_{i,j}(0)|\gg1$ or large separation of solitons. By introducing the Plank constant $\varepsilon$, we can prove the maximum time tends to infinity as $\varepsilon\to 0$ provided the barycenters of $k$ solitons do not collide.

Our paper is organized as follows: After giving the definition of multi-soliton manifold $\mathbf{M}_k$ and parameters we need in Section 3 and 4, we will prove that the skew orthogonal decomposition is unique for $\psi$ satisfying a couple of assumptions in Section 5. Then, we calculate derivative with respect to time of each variable to give the modulation formula for solitons dynamics in Section 6. The necessary estimate for time derivative of Lyapunov functional is given in Section 7. In Section 8, we finish our proof by showing that the $H^1$ norm of $\omega(x,t)$ is small for every $t$ such that $\psi$ is in the neighborhood $U_\delta$ of $\mathbf{M}_k$.

\section{Assumptions and the main results}
In this section, we first formulate a couple of assumptions on the potential $V_\varepsilon(x)$, the nonlinearity $f$ and the initial data $\psi_0$. These assumptions are important since even a small change of variables can bring big influence to the multi-soliton dynamics. 
\begin{itemize}
	\item[(A)] (Potential) The external potential $V_\varepsilon(x)\equiv V(\varepsilon_vx)$, where $V(x)\in C^2(\mathbb{R}^N)$ and $\varepsilon_v=\varepsilon^h$ for some $h>2$, that is, for any $\alpha\in (0,2]$
	 $$|\partial_x^\alpha V_\varepsilon(x)|\le C_\alpha \varepsilon_v^{|\alpha|},$$
	 where $C_\alpha$ is a positive constant depending on $\alpha$.
	\item[(B)] (Nonlinearity) $\forall s\in \mathbb{C}$, $f(s)=\lambda |s|^{p-1}s$, with $1<p<1+\frac{4}{N}$ and $\lambda>0$. 
	\item[(C)] (Initial data) Denote 
	$$\sigma_i(t):=(a_i,v_i,\gamma_i,\mu_i)\in\mathbb{R}^N\times\mathbb{R}^N\times\mathbb{R}/(0,2\pi]\times\mathbb{R}$$
	and for $i=1,\cdots,k$
	$$\tilde\Psi_{\sigma_i,\varepsilon}:=e^{\mathbf{i}(\frac{1}{2}v_i\cdot(x-a_i)+\gamma_i)}\eta_{\mu_i,\varepsilon}(x-a_i).$$
	The initial data $\psi(0,x)=\psi_0(x)$ satisfies
	$$||\psi_0(x)-\sum_{l=1}^{k}\tilde\Psi_{\sigma_{l,0},\varepsilon}(x)||_{H^1}^2\le c\varepsilon_v^2,$$
	where $\varepsilon_v$ is the same as in assumption (A). In addition, we assume that there are positive constants $K,L,\mu_{inf},\mu_{sup}$ with $\mu_{inf}<\mu_{sup}$ such that the components of $\sigma_{i,0}$ ($i=1,\cdots,k$) satisfy $|a_{j,0}-a_{i,0}|>6L$ with $j\neq i$; $|v_{i,0}|< K$; $\mu_{i,0} \subset I$ with $I\subset (\mu_{inf},\mu_{sup})$.
\end{itemize}

We give a brief explaination to these assumptions. Assumption (A) states that the length scale of potential $V_\varepsilon(x)$ is of higher order to the size scale of $\psi$, namely, the potential $V_\varepsilon(x)$ varies slowly in space as compared to the solitons. This guarantees that the velocity change of each soliton is small compared to their initial velocity, so that we can obtain the solitons dynamics over a large time interval.

By assumption (B), the Cauchy problem \eqref{1-1} and \eqref{1-2} is well-posed in $H^1$ and $H^2$. There exists an energy functional $F: H^1(\mathbb{R}^N, \mathbb{C})\to \mathbb{R}$ corresponding to the nonlinearity $f$, which can be written in an explicit form:
\begin{equation}\label{2-1}
F(\psi)=\int_{\mathbb{R}^N}|\psi|^{p+1}dx.
\end{equation}
Obviously, $F(\psi)$ is maintained by a translation $\mathcal{T}^{tr}_a: \psi(x) \to \psi(x-a) \ \  \forall a\in \mathbb{R}^N$, a rotation $\mathcal{T}^r_R: \psi(x)\to \psi(R^{-1}x)$, a gauge transform $\mathcal{T}^g_\gamma: \psi(x)\to e^{\mathbf{i}\gamma}\psi(x) \ \ \forall \gamma \in \mathbb{R}/[0,2\pi)$, a boost transform $\mathcal{T}^b_v: \psi(x)\to e^{\frac{\mathbf{i}}{2}v\cdot \frac{x}{\varepsilon}}\psi(x) \ \ \forall v\in \mathbb{R}^N$, and a complex conjugation $\mathcal{T}^{c}:\psi(x)\to\bar{\psi}(x)$. Since $p<1+\frac{4}{n}$, the linearized operator $\mathcal{L}_{\eta_{\mu,\varepsilon}}$ at $\eta_{\mu,\varepsilon}$ has only one negative eigenvalue, which can be found in the appendix of \cite{Fr3}. In \cite{Gri1,Gri2}, this fact was used to prove the orbital stability of solutons generated by $\eta_{\mu,\varepsilon}$. In Section 4, we will also show that the null space of $\mathcal{L}_{\eta_{\mu,\varepsilon}}$ has exact $N+1$ bases under this assumption.

Assumption (C) claims that the initial data $\psi_0(x)$ is sufficiently close to the superposition of $k$ solitons. Moreover, $k$ peaks can not get too close to each other at the beginning, so that the energy caused by the interaction of two peaks is a small term of exponential decay rate. We also require that the initial velocity has an upper bound $K$, and the parameters $\mu_i>0$ are in a common range in $\mathbb{R}^+$.

Now, we are prepared to state our main result:
\begin{theorem}\label{thm1}
	Assume that assumptions (A) (B) and (C) are satisfied. Then there is a constant $T_0>0$ (independent of $\varepsilon$ and $\varepsilon_v$ but possibly dependent on constants $L,K$ in assumption (C)), such that for $t\le \min\{\frac{T_0}{\varepsilon},\frac{L}{K}\}$, the solution to \eqref{1-1} and \eqref{1-2} has the form
	\begin{equation}\label{2-2}
	\psi(x,t)=\sum_{l=1}^{k}\tilde\Psi_{\sigma_l,\varepsilon}+w,
	\end{equation}
	where $||w||_{H^1}=O(\varepsilon_v)$, $\tilde\Psi_{\sigma_l,\varepsilon}$ is defined in assumption (C). The corresponding parameters $a_i,v_i,\gamma_i,\mu_i$ in $\sigma_i$ only depend on $t$, and satisfy the differential equations
	\begin{equation}\label{2-3}
	\frac{1}{2}\dot{v_i}=-\nabla V_\varepsilon(a_i)+O(\frac{\varepsilon_v^2}{\varepsilon})
	\end{equation}
	\begin{equation}\label{2-4}
	\dot{a_i}=v_i+O(\varepsilon\cdot\varepsilon_v^2)
	\end{equation}
	\begin{equation}\label{2-5}
	\varepsilon\dot{\gamma_i}=\mu_i-V_\varepsilon(a_i)-\frac{1}{4}v_i^2+O(\varepsilon_v^2)
	\end{equation}
	\begin{equation}\label{2-6}
	\dot{\mu_i}=O(\frac{\varepsilon_v^2}{\varepsilon}),
	\end{equation}
	with $i=1,\cdots,k$.
\end{theorem} 

\begin{remark}
	Let us consider the NSE with Hartree nonlinearity $f=(W\star |\psi|^2)\psi$, where $W$ is in $L^q\cap L^\infty$ with $q>\frac{N}{2}$, $q\ge 1$, and $W$ decays exponentially fast. Since the ground states is of similar form as the NSE with local nonlinearities, we can use a same argument to prove the estimate for $w$ and solitons dynamics in Theorem \ref{thm1} still hold for this case.
\end{remark}

\begin{remark}
	Since the consevation of energy will be used to prove Theorem \ref{thm1}, we assume the potential $V_\varepsilon$ independent of $t$. However, for the case $V_\varepsilon(t,x)\equiv V(t,\varepsilon_vx)$ with $V(t,x)\in W^{1,\infty}(\mathbb{R},C^2(\mathbb{R}^N))$ as in \cite{Ab}, we believe that by imposing appropriate restrictions on $\partial_t V_\varepsilon$ and using the transport formula for energy, a similar dynamic can be obtained for this multi-soliton system.
\end{remark}

As a consequence of Theorem \ref{thm1}, we obtain the integral form of solitons dynamics:
\begin{corollary}\label{coro1}
	Assume that assumptions (A) (B) and (C) are satisfied. Then for $t\le \min\{\frac{T_0}{\varepsilon},\frac{L}{K}\}$ the same as in Theorem \ref{thm1}, the velocity $v_i(t)$, location $a_i(t)$ and phase $\gamma_i(t)$ of soliton $\Psi_{\sigma_i,\varepsilon}$ satisfy the integral equations
	\begin{equation}\label{2-7}
	v_i(t)=v_i(0)+2\int_0^t\nabla V_\varepsilon(a_i(s))ds+O(\frac{\varepsilon_v^2}{\varepsilon^2}),
	\end{equation}
	\begin{equation}\label{2-8}
	a_i(t)=a_i(0)+\int_0^tv_i(s)ds+O(\varepsilon_v^2),
	\end{equation}
	\begin{equation}\label{2-9}
	\gamma_i(t)=\gamma_i(0)+\frac{1}{\varepsilon}\int_0^t \big(\mu_i(s)-V_\varepsilon(a_i(s))-\frac{1}{4}v_i(s)^2\big)ds+O(\frac{\varepsilon_v^2}{\varepsilon^2}),
	\end{equation}
	with $i=1,\cdots,k$, and for $\mu_i(t)$ it holds
	\begin{equation}\label{2-10}
	\mu_i(t)=\mu_i(0)+O(\frac{\varepsilon_v^2}{\varepsilon^2}), \ \ \ \ \ \forall i=1,\cdots,k.
	\end{equation}
    In addition, it holds
    \begin{equation}\label{2-11}
    	|\sigma_{i,0}-\sigma_i(0)|\le C\varepsilon_v
    \end{equation}
    for $i=1,\cdots,k$.
\end{corollary}

We can imopose following additional assumption on the initial data, so that the maximum time can be extended to $\frac{T_0}{\varepsilon}$:
\begin{itemize}
	\item[(D)] (Barycenters non-coinciding) Denote $\tilde{a}_i(t)=a_{i,0}+tv_{i,0}$ for $i=1,\cdots,k$. For every fixed $s\in [0,+\infty)$ and $i,j$ satisfying $1\le i,j\le k$ with $i\neq j$, it holds
	\begin{equation}\label{2-12}
	|\tilde{a}_i(s)-\tilde{a}_j(s)|>2d,
	\end{equation}
    where $d$ is a positive constant.
\end{itemize}

We have following theorem:
\begin{theorem}\label{thm2}
	Assume that assumptions (A) (B) (C) and (D) are satisfied. Then the maximum time in Theorem \ref{thm1} can be extended to $\frac{T_0}{\varepsilon}$. 
\end{theorem}

\section{Variational structure and conserved quantities of NLS equations}

We consider the modified nonlinear Schr\"{o}dinger equation \eqref{1-3} with time parameter not scaled:
\begin{equation*}
\begin{cases}
\mathbf{i}\partial_t\psi=-\varepsilon^2\Delta\psi+V_\varepsilon(x)\psi-f(\psi),\\
\psi(0,x)=\psi_0(x),\\
\end{cases}
\end{equation*}
and study its solution in the space $H^1(\mathbb{R}^N,\mathbb{C})$, which can be considered as a direct sum of two real space $H^1(\mathbb{R}^N,\mathbb{R})\oplus H^1(\mathbb{R}^N,\mathbb{R})$ under the identification $\psi \to (\mathbf{Re}\psi, \mathbf{Im}\psi)$. We equip this complex Sobolev space with the symplectic form
\begin{equation}\label{3-1}
\omega_\varepsilon(u_1,u_2):=\mathbf{Im}\int_{\mathbb{R}^N}\frac{1}{\varepsilon^N}u_1\bar{u}_2dx.
\end{equation}
The space $H^1(\mathbb{R}^N,\mathbb{C})$ also has a real inner product:
\begin{equation}\label{3-2}
\langle u_1,u_2\rangle_\varepsilon:=\mathbf{Re}\int_{\mathbb{R}^N}\frac{1}{\varepsilon^N}u_1\bar{u}_2dx,
\end{equation}
which induces the $L^2$ norm $||u||_{L^2}^2=\langle u,u\rangle_\varepsilon$. With the identification $\psi \to (\mathbf{Re}\psi, \mathbf{Im}\psi)$, the operator of multiplication by $\mathbf{i}^{-1}$ is identified with the operator
\begin{equation}\label{3-3}
J= 
\left(\begin{array}{cccc} 
 0  & 1\\ 
-1 &  0\\ 
\end{array}\right)
\end{equation}
Thus we have $\omega_\varepsilon(u_1,u_2)=\langle u_1,J^{-1}u_2\rangle_\varepsilon$.

The Hamiltonian functional related to \eqref{1-3}  on $H^1(\mathbb{R}^N,\mathbb{C})$ is
\begin{equation}\label{3-4}
\mathcal{H}_{V,\varepsilon}(\psi):=\frac{1}{2}\int_{\mathbb{R}^N}\big(\frac{1}{\varepsilon^{N-2}}|\nabla\psi|^2+\frac{1}{\varepsilon^N}V_\varepsilon|\psi|^2dx-\frac{1}{\varepsilon^N}F(\psi)\big)dx,
\end{equation}
where $F(\psi)$ is defined in Condtion (B) in Section 2. Since the Hamiltonian $\mathcal{H}_{V,\varepsilon}(\psi)$ is autonomous and invariant under the gauge transformation ($\mathcal{H}_{V,\varepsilon}(e^{\mathrm{i}\gamma}\psi)=\mathcal{H}_{V,\varepsilon}(\psi)$), under the evolution of \eqref{1-3} we have conservation of energy:
\begin{equation}\label{3-5}
\frac{d}{dt}\mathcal{H}_{V,\varepsilon}=0
\end{equation}
and conservation of mass
\begin{equation}\label{3-6}
\frac{d}{dt}(\frac{1}{2}\int_{\mathbb{R}^N}\frac{1}{\varepsilon^N}|\psi|^2dx)=0.
\end{equation}

Notice that we do not scale the time parameter in \eqref{1-3}, so the initial velocity condition for the Cauchy problem \eqref{1-3} is 
\begin{equation}\label{3-7}
|v_i(0)| < K \varepsilon \ \ \ \ \ \forall i=1,\cdots,k.
\end{equation} 
For the potential term $\int_{\mathbb{R}^N}\frac{1}{\varepsilon^N}V_\varepsilon|\psi|^2dx$, we have
\begin{equation}\label{3-8}
\frac{1}{\varepsilon}\frac{d}{dt}(\frac{1}{2}\int_{\mathbb{R}^N}\frac{1}{\varepsilon^N}V_\varepsilon|\psi|^2dx)=\langle (\nabla V_\varepsilon)\mathrm{i}\psi, \varepsilon\nabla\psi\rangle_\varepsilon,
\end{equation}
which can be found in \cite{Fr3}. And for $D\subset \mathbb{R}^N$ whose boundary is smooth, the rate of mass exchange on $\partial D$ is dominated by a flux term, which can be rewritten as an interior divergent form by Green's theorem:
\begin{equation}\label{3-9}
\frac{1}{\varepsilon}\frac{d}{dt}(\frac{1}{2}\frac{1}{\varepsilon^N}|\psi|^2)=-\nabla\cdot \mathbf{Re}(\frac{1}{\varepsilon^{N-1}}\mathbf{i}\psi\nabla \psi).
 \end{equation}
Similarly, for momentum we have
\begin{equation}\label{3-10}
\frac{1}{\varepsilon}\frac{d}{dt}\big(\mathbf{Re}(\frac{1}{\varepsilon^N}\mathrm{i}\psi\cdot\varepsilon\overline{\nabla\psi})\big)=-\frac{1}{\varepsilon^N}\nabla V_\varepsilon|\psi|^2+\frac{1}{\varepsilon^N}\nabla\cdot \mathbb{T},
\end{equation}
where $\mathbb{T}$ is the stress tensor with the property
$$|\mathbb{T}|\le C(\varepsilon^2|\nabla \psi|^2+\varepsilon^2|\Delta|\psi^2||+|\psi|^2+|\psi|^{p+1}),$$
see \cite{B2}. For more information about conserved quantities of NSE, we refer to \cite{Br} and the references therein.

Solutions to the NSE are various with different physical explanations. In particular, we are concerned with $\eta_{\mu,\varepsilon}$, which is radial decreasing and described as local minimizers of the Hamiltonian $\mathcal{H}_{V\equiv0,\varepsilon}(\psi)$ restricted to the spheres
\begin{equation}\label{3-11}
\{\psi\in H^1(\mathbb{R}^N,\mathbb{C}): \frac{1}{2}\int_{\mathbb{R}^N}\frac{1}{\varepsilon^N}|\psi|^2dx=m\},
\end{equation}
for $m>0$. Thus $\eta_{\mu,\varepsilon}$ is a critical point of
\begin{equation}\label{3-12}
\mathcal{E}_{\mu,\varepsilon}:=\frac{1}{2}\int_{\mathbb{R}^N}\big(\frac{1}{\varepsilon^{N-2}}|\nabla\psi|^2+\frac{1}{\varepsilon^N}\mu|\psi|^2dx-\frac{1}{\varepsilon^N}F(\psi)\big)dx
\end{equation}
with $\mu=\mu(m)$ the Lagrange multiplier, known as the ground state solution. The linearized operator corressponding to $\eta_{\mu,\varepsilon}$ possesses very good properties such as non-degeneracy and coercivity, which we will discuss in following sections.

We end this section by giving an estimate associated with Taylor expansion of $F(\psi)$ at $\eta$. Using assumption (B), if we define
\begin{equation}\label{3-13}
R_\eta^{(2)}(w):=\frac{1}{\varepsilon^n}F(\eta+w)-\frac{1}{\varepsilon^n}F(\eta)-\langle F'(\eta), w\rangle_\varepsilon,
\end{equation}
\begin{equation}\label{3-14}
R_\eta^{(3)}(w):=\frac{1}{\varepsilon^n}F(\eta+w)-\frac{1}{\varepsilon^n}F(\eta)-\langle F'(\eta), w\rangle_\varepsilon-\frac{1}{2}\langle F''(\eta)w, w\rangle_\varepsilon,
\end{equation}
\begin{equation}\label{3-15}
N_\eta(w):=F'(\eta+w)-F'(\eta)-F''(\eta)w,
\end{equation}
with $\eta\in H^2(\mathbb{R}^N,\mathbb{C})$ and $w\in H^1(\mathbb{R}^N,\mathbb{C})$, then we have
\begin{equation}\label{3-16}
|R_\eta^{(2)}(w)|\le C||w||_{H^1}^2, \ \ |R_\eta^{(3)}(w)|\le C||w||_{H^1}^3, \ \ ||N_\eta(w)||_{H^{-1}}\le C||w||_{H^1}^2,
\end{equation}
where $C$ is a constant dependent only on $||\eta||_{H^1}$ provided $||w||_{H^1}\le 1$.

\section{The manifold of solitons}

To introduce the manifold of solitons, we first focus on the linearized operator:
\begin{equation}\label{4-1}
\mathcal{L}_{\eta_{\mu,\varepsilon}}:=-\varepsilon^2\Delta+\mu-f'(\eta_{\mu,\varepsilon}).
\end{equation}
According to the Appendix C of \cite{Fr3}, $\mathcal{L}_{\eta_{\mu,\varepsilon}}$ has only one negative eigenvalue, and the null space of $\mathcal{L}_{\eta_{\mu,\varepsilon}}$ is spanned by $N+1$ vectors:
\begin{equation}\label{4-2}
N(\mathcal{L}_{\eta_{\mu,\varepsilon}})=span\{\mathbf{i}\eta_{\mu,\varepsilon}, \partial_{x_1}\eta_{\mu,\varepsilon},\cdots, \partial_{x_N}\eta_{\mu,\varepsilon}\}.
\end{equation}
We will show that the tangent space of soliton manifold are spanned by these $N+1$ zero modes and other $N+1$ symplectically associated zero modes.

Using those transforms defined in Section 2, we can give the combined symmetry transform $\mathcal{T}_{av\gamma}$:
\begin{equation}\label{4-3}
\psi(x)\to \mathcal{T}_{av\gamma}\psi(x)= \mathcal{T}^{tr}_a \mathcal{T}^b_v \mathcal{T}^g_\gamma\psi(x)=e^{\frac{\mathbf{i}}{2}v\cdot\frac{x-a}{\varepsilon}+\gamma}\psi(x-a).
\end{equation}
Denote that $\Psi_{\sigma,\varepsilon}:=\mathcal{T}_{av\gamma}\eta_{\mu,\varepsilon}$. We define the manifold of a single soliton as
\begin{equation}\label{4-4}
\mathbf{M}:=\{\Psi_{\sigma,\varepsilon}: (a,v,\gamma,\mu)\in \mathbb{R}^N\times\mathbb{R}^N\times (\mathbb{R}/(0,2\pi]\times I\},
\end{equation}
We calculate the derivatives of $\mathcal{T}_{av\gamma}\eta_{\mu,\varepsilon}$ at origin to obtain the tangent space $\mathbf{T}_{\Psi_{\sigma_0,\varepsilon}}\mathbf{M}$ with $\sigma_0=(0,0,0,\mu)$, which is given by
\begin{equation}\label{4-5}
(z_t,z_b,z_g,z_s)\in \mathbb{R}^N\times\mathbb{R}^N\times\mathbb{R}\times\mathbb{R},
\end{equation}
where
\begin{equation}\label{4-6}
z_t:=\varepsilon\nabla_a\mathcal{T}^{tr}_a\eta_{\mu,\varepsilon}\big{|}_{a=0}=-\varepsilon\nabla\eta_{\mu,\varepsilon},
\end{equation}
\begin{equation}\label{4-7}
z_b:=2\nabla_v\mathcal{T}^{b}_v\eta_{\mu,\varepsilon}\big{|}_{v=0,t=0}=\frac{\mathbf{i}x}{\varepsilon}\eta_{\mu,\varepsilon},
\end{equation}
\begin{equation}\label{4-8}
z_g:=\partial_\gamma\mathcal{T}^{g}_\gamma\eta_{\mu,\varepsilon}\big{|}_{\gamma=0}=\mathbf{i}\eta_{\mu,\varepsilon},
\end{equation}
\begin{equation}\label{4-9}
z_s:=\partial_\mu\eta_{\mu,\varepsilon}.
\end{equation}
According to \eqref{4-2}, we have
\begin{equation}\label{4-10}
\mathcal{L}_{\eta_{\mu,\varepsilon}}z_t=0 \ \ \ and \ \ \ \mathcal{L}_{\eta_{\mu,\varepsilon}}z_g=0.
\end{equation}
By taking the derivatives of equation $\mathcal{E}_{\mu,\varepsilon}(\mathcal{T}^b_v\eta_{\mu,\varepsilon})=0$ with respect to $v$ and $\mu$, we have
\begin{equation}\label{4-11}
\mathcal{L}_{\eta_{\mu,\varepsilon}}z_b=2\mathbf{i}z_t \ \ \ and \ \ \ \mathcal{L}_{\eta_{\mu,\varepsilon}}z_s=\mathbf{i}z_g.
\end{equation}
Thus $z_b$ and $z_s$ are in the null space of the operator $(\mathbf{i}\mathcal{L}_{\eta_{\mu,\varepsilon}})^2$:
\begin{equation}\label{4-12}
\mathbf{i}\mathcal{L}_{\eta_{\mu,\varepsilon}}(\mathbf{i}\mathcal{L}_{\eta_{\mu,\varepsilon}}z_b)=0 \ \ \ and \ \ \ \mathbf{i}\mathcal{L}_{\eta_{\mu,\varepsilon}}(\mathbf{i}\mathcal{L}_{\eta_{\mu,\varepsilon}}z_s)=0.
\end{equation}
Denote $\{z_1,\cdots,z_{2N+2}\}:=\{z_t,z_b,z_g,z_s\}$. The tangent space $\mathbf{T}_{\Psi_{\sigma_0,\varepsilon}}\mathbf{M}$ then inherits a symplectic form $(H^1, \omega_\varepsilon)$, which is determined by the operator
\begin{equation}\label{4-13}
\omega_{\Psi_{\sigma_0,\varepsilon}}(u,v):=\langle u,\Omega_{\eta_{\mu,\varepsilon}}v\rangle_\varepsilon, \ \ \ \forall u,v\in\mathbf{T}_{\Psi_{\sigma_0,\varepsilon}}\mathbf{M}
\end{equation}
where $\Omega_{\eta_{\mu,\varepsilon}}$ is a $(2N+2)\times(2N+2)$ matrix whose component is given by $\Omega_{\eta_{\mu,\varepsilon}}^{i,j}=\langle z_i,J^{-1}z_j\rangle_\varepsilon$. We are going to show that this symplectic form is non-degenerate. For this purpose, we denote
\begin{equation}\label{4-14}
m(\mu):=\frac{1}{2}\int_{\mathbb{R}^N}\frac{1}{\varepsilon^N} |\eta_{\mu,\varepsilon}|^2dx.
\end{equation}

\begin{lemma}\label{lem4-1}
	If $\mu\in I$, then $\Omega_{\eta_{\mu,\varepsilon}}$ is invertible. Consequently $\forall z \in \mathbf{T}_{\Psi_{\sigma_0,\varepsilon}}\mathbf{M}$ with $\sigma_0=(0,0,0,\mu)$, there exists at least one element $z'\in \mathbf{T}_{\Psi_{\sigma_0,\varepsilon}}\mathbf{M}$, such that $\omega_\varepsilon(z,z')\neq 0$.
\end{lemma}
\begin{proof}
	We compute each component $\Omega_{\eta_{\mu,\varepsilon}}^{m,n}=\langle z_m,J^{-1}z_n\rangle \ , \ 1\le m,n\le 2N+2$, so the matrix $\Omega_{\eta_{\mu,\varepsilon}}$ can be written as
	\begin{equation}\label{4-15}
	\Omega_{\eta_{\mu,\varepsilon}}= 
	\left(\begin{array}{cccc} 
	0 & -m\mathbf{1} & 0 & 0\\ 
	m\mathbf{1} & 0 & 0 & 0\\ 
	0 & 0 & 0 & m'(\mu)\\
	0 & 0 & -m'(\mu) & 0\\
	\end{array}\right)
	\end{equation}
	with $\mathbf{1}$ the $N\times N$ matrix where the elements on diagonal are $1$, while others are $0$. By assumption (B) in Section 2, $\forall \mu\in I\subset \mathbb{R}^+$ we have
	\begin{equation}\label{4-16}
	\begin{split}
	m'(\mu)&=\partial_\mu\int_{\mathbb{R}^N}\frac{1}{\varepsilon^N}\eta_{\mu,\varepsilon}^2dx=\partial_\mu\big(\mu^{\frac{2}{p-1}}\int_{\mathbb{R}^N}\frac{1}{\varepsilon^N}\eta_{1,\varepsilon}(\sqrt{\mu}x)^2dx\big)\\
	&=\partial_\mu\big(\mu^{\frac{2}{p-1}-\frac{N}{2}}\int_{\mathbb{R}^N}\frac{1}{\varepsilon^N}\eta_{1,\varepsilon}(x')^2dx'\big)\\
	&=\partial_\mu(\mu^{\frac{2}{p-1}-\frac{N}{2}})\int_{\mathbb{R}^N}\frac{1}{\varepsilon^N}\eta_{1,\varepsilon}(x')^2dx'\\
	&>0
	\end{split}
	\end{equation}
	since $p<1+\frac{4}{N}$. Hence the matrix $\Omega_{\eta_{\mu,\varepsilon}}$ is invertible and the corresponding symplectic form is non-degenerate. The rest of the conclusion of this lemma follows as a consequence.
\end{proof}

When the transform $\mathcal{T}_{av\gamma}$ acts on the soliton $\eta_{\mu,\varepsilon}$, the tangent space $\mathbf{T}_{\Psi_{\sigma,\varepsilon}}\mathbf{M}$ at $\Psi_{\sigma,\varepsilon}$ is given by
\begin{equation}\label{4-17}
(\varepsilon\nabla_a\Psi_{\sigma,\varepsilon},2\nabla_v\Psi_{\sigma,\varepsilon},\partial_\gamma\Psi_{\sigma,\varepsilon},\partial_\mu\Psi_{\sigma,\varepsilon})\in \mathbb{R}^N\times\mathbb{R}^N\times\mathbb{R}\times\mathbb{R},
\end{equation}
where
\begin{equation}\label{4-18}
\varepsilon\nabla_a\Psi_{\sigma,\varepsilon}=-\mathcal{T}_{av\gamma}(\varepsilon\nabla\eta_{\mu,\varepsilon})-\frac{v}{2}\mathcal{T}_{av\gamma}(\mathbf{i}\eta_{\mu,\varepsilon})=\mathcal{T}_{av\gamma}z_t-\frac{v}{2}\mathcal{T}_{av\gamma}z_g,
\end{equation}
\begin{equation}\label{4-19}
2\nabla_v\Psi_{\sigma,\varepsilon}=\mathcal{T}_{av\gamma}(\frac{\mathbf{i}x}{\varepsilon}\eta_{\mu,\varepsilon})=\mathcal{T}_{av\gamma}z_b-\frac{a}{\varepsilon}\mathcal{T}_{av\gamma}z_g,
\end{equation}
\begin{equation}\label{4-20}
\partial_\gamma\Psi_{\sigma,\varepsilon}=\mathcal{T}_{av\gamma}(\mathbf{i}\eta_{\mu,\varepsilon})=\mathcal{T}_{av\gamma}z_g,
\end{equation}
\begin{equation}\label{4-21}
\partial_\mu\Psi_{\sigma,\varepsilon}=\mathcal{T}_{av\gamma}(\partial_\mu\eta_{\mu,\varepsilon})=\mathcal{T}_{av\gamma}z_s.
\end{equation}
Note that the operator $\mathcal{T}_{av\gamma}$ is canonical, which leaves $\langle\cdot,J^{-1}\cdot\rangle_\varepsilon$ unchanged, that is
\begin{equation}\label{4-22}
\langle \mathcal{T}_{av\gamma}u,J^{-1}\mathcal{T}_{av\gamma}v\rangle_\varepsilon=\langle u,J^{-1}v\rangle_\varepsilon.
\end{equation}
So the matrix $\Omega_{\Psi_{\sigma,\varepsilon}}$ induced by \eqref{4-18}-\eqref{4-21} is related to $\Omega_{\eta_{\mu,\varepsilon}}$ by a similarity transform. 

\section{skew orthogonal decomposition}

In this section, we will first introduce the skew orthogonal decomposition of $\psi$ along the multi-soliton manifold $\mathbf{M}_k$ and the skew orthogonal direction, and then prove this decomposition is unique when $\psi$ is sufficiently close to $\mathbf{M}_k$ (see Arnol'd \cite{A}).

Let
\begin{equation}\label{5-1}
	\Sigma_j:=B_L(a_{j,0})\times B_{K\varepsilon}(0)\times (\mathbb{R}/(0,2\pi])\times I
\end{equation}
for $j=1,\cdots,k$ and
$$\boldsymbol{\Sigma}=\Sigma_1\times\cdots\times\Sigma_k.$$
We define the $\delta$-neighborhood  
\begin{equation}\label{5-2}
U_\delta=\{\psi\in H^1(\mathbb{R}^N,\mathbb{C}): \inf\limits_{\boldsymbol{\sigma}\in\boldsymbol{\Sigma}}||\psi-\sum_{l=1}^k\Psi_{\sigma_l,\varepsilon}||_{H^1}\le\delta\}
\end{equation}
of the multi-soliton manifold 
\begin{equation}\label{5-3}
\mathbf{M}_k:=\{\sum\limits_{l=1}^k\Psi_{\sigma_l,\varepsilon}:\boldsymbol{\sigma}\in \boldsymbol{\Sigma}=(\sigma_1,\cdots,\sigma_k)\}.
\end{equation}

For simplicity, we use the notations
\begin{equation}\label{5-4}
\sigma_j:=(\sigma_{j,1},\cdots,\sigma_{j,2N+2}):=(a_j,v_j,\gamma_j,\mu_j),
\end{equation}
\begin{equation}\label{5-5}
\begin{split}
(z_{j,1},\cdots,z_{j,2N+2})=\big(-\mathcal{T}_{a_j,v_j,\gamma_j}&(\varepsilon\nabla\eta_{\mu_j,\varepsilon}),\mathcal{T}_{a_j,v_j,\gamma_j}(\frac{\mathrm{i}x}{\varepsilon}\eta_{\mu_j,\varepsilon}),\\
&\mathcal{T}_{a_j,v_j,\gamma_j}(\mathrm{i}\eta_{\mu_j,\varepsilon}),\mathcal{T}_{a_j,v_j,\gamma_j}(\partial_\mu\eta_{\mu_j,\varepsilon})\big)
\end{split}
\end{equation}
with $j=1,\cdots,k$. We have following lemma, which claims that the interaction among $k$ solitons is a small term.
\begin{lemma}\label{lem5-1}
	For two bases $z_{i,m}$ and $z_{j,n}$ satisfying $i\neq j$, it holds
	\begin{equation}\label{5-6}
	|\omega_\varepsilon(z_{i,m},z_{j,n})|=O(e^{-\frac{c}{\varepsilon}}),
	\end{equation}
	with $c$ a positive constant.
\end{lemma}
\begin{proof}
	According to \eqref{1-8}, we have
	\begin{equation}\label{5-7}
	|z_{i,m}|\le Ce^{-\frac{c_1|x-a_i|}{\varepsilon}} \ \ \ and \ \ \ |z_{j,n}|\le Ce^{-\frac{c_2|x-a_j|}{\varepsilon}},
	\end{equation}
	where $a_i\in B_L(a_{i,0})$ and $a_j\in B_L(a_{j,0})$. By assumption (C) in Section 2, it holds $|a_i-a_j|> L$. If we denote $x'=(x-a_i)/\varepsilon$, we can deduce that
	\begin{equation}\label{5-8}
	\begin{split}
	|\omega_\varepsilon(z_{i,m},z_{j,n})|&\le C\int_{\mathbb{R}^N}\frac{1}{\varepsilon^N}e^{-\frac{c_1|x-a_i|}{\varepsilon}}e^{-\frac{c_2|x-a_j|}{\varepsilon}}dx\\
	&=C\int_{\mathbb{R}^N}e^{-c_1x'}e^{-c_2|x'-\frac{a_i-a_j}{\varepsilon}|}dx'\\
	&\le Ce^{-\frac{c}{\varepsilon}}
	\end{split}
	\end{equation}
	Hence the proof is complete.
\end{proof}

The next proposition is the main result of this section, which shows the skew orthogonal decomposition is unique in $U_\delta$ provided $\delta$ is small enough. 
\begin{proposition}\label{prop5-2}
If $\delta\ll \inf_{\mu\in I}m'(\mu)$, then there exists a unique $\boldsymbol{\sigma}=\boldsymbol{\sigma}(\psi)\in C^1(U_\delta,\boldsymbol{\Sigma})$ such that
\begin{equation}\label{5-9}
\omega_\varepsilon(\psi-\sum_{l=1}^k\Psi_{\sigma_j,\varepsilon},z)=0, \ \ \ \forall z\in \mathbf{T}_{\Psi_{\sigma_l,\varepsilon}}\mathbf{M} \ \ \ l=1,\cdots,k
\end{equation}
\end{proposition}
\begin{proof}
	We denote 
	\begin{equation}\label{5-10}
	\partial_{i,n}=\varepsilon\partial_{\sigma_{i,n}}+\frac{v_{i,n}}{2}\partial_{\sigma_{i,2N+1}} \ \  for \ \ i=1,\cdots,k, \ \ \ n=1,\cdots,N,
	\end{equation}
    \begin{equation}\label{5-11}
    \partial_{i,n}=2\partial_{\sigma_{i,n}}+\frac{a_{i,n}}{\varepsilon}\partial_{\sigma_{i,2N+1}} \ \  for \ \ i=1,\cdots,k, \ \ \ n=N+1,\cdots,2N,
    \end{equation}
	and
	\begin{equation}\label{5-12}
	\partial_{i,n}=\partial_{\sigma_{i,n}} \ \  for \ \ i=1,\cdots,k, \ \ \ n=N+1,2N+2.
	\end{equation}
	Then we can apply the implicit function theorem for $G: H^1(\mathbb{R}^N,\mathbb{C})\to \mathbb{R}^{k(2N+2)}$ defined by
	\begin{equation*}
	G_{j,m}(\psi,\boldsymbol{\sigma}):=\langle \psi-\sum_{l=1}^k\Psi_{\sigma_l,\varepsilon},J^{-1}z_{j,m}\rangle_\varepsilon \ \ \ \forall j=1,\cdots,k \ and \ \forall m=1,\cdots,2N+2.
	\end{equation*}
	Since both $\Psi_{\sigma_l,\varepsilon}$ and $z_{j,m}$ are  $C^1$, and $G$ is linear in $\psi$, we see that $G$ is $C^1$ in $\boldsymbol{\sigma}$. It is obvious that 
	$G_{j,m}(\sum_{l=1}^k\Psi_{\sigma_l,\varepsilon},\boldsymbol{\sigma}_0)=0$ for any $\boldsymbol{\sigma}_0=(\sigma_1,\cdots,\sigma_k)\in\boldsymbol{\Sigma}$. We only need to verify that $\partial_{i,n}G_{j,m}(\sum_{l=1}^k\Psi_{\sigma_l,\varepsilon},\boldsymbol{\sigma})\big|_{\boldsymbol{\sigma}=\boldsymbol{\sigma}_0}$ is invertible.
	
	By \eqref{5-10} \eqref{5-11} and \eqref{5-12}, it holds
	\begin{equation}\label{5-13}
	\partial_{i,n}G_{j,m}(\sum_{l=1}^k\Psi_{\sigma_l,\varepsilon},\boldsymbol{\sigma})\big|_{\boldsymbol{\sigma}=\boldsymbol{\sigma}_0}=-\langle z_{i,n}, J^{-1}z_{j,m}\rangle_\varepsilon
	\end{equation}
	Using the canonical property \eqref{4-22} of $\mathcal{T}_{av\gamma}$ , we have
	\begin{equation}\label{5-14}
	\langle z_{i,n}, J^{-1}z_{i,m}\rangle_\varepsilon=\Omega_{\eta_{\mu_i,\varepsilon}}^{n,m}.
	\end{equation}
	By Lemma \ref{lem5-1}, For $i\neq j$ we have
	\begin{equation}\label{5-15}
	\langle z_{i,n}, J^{-1}z_{j,m}\rangle_\varepsilon=O(e^{-\frac{c}{\varepsilon}}).
	\end{equation}
	($z_{i,n}$ and $z_{j,m}$ with $i\neq j$ are skew orthogonal in an asymptotic sense.) Thus using Lemma \eqref{lem4-1}, we have that $\partial_{i,n}G_{j,m}(\sum_{l=1}^k\Psi_{\sigma_{0,l},\varepsilon},\boldsymbol{\sigma})\big|_{\boldsymbol{\sigma}=\boldsymbol{\sigma}_0}$ is invertible for all $\boldsymbol{\sigma}_0$ when $\varepsilon$ is sufficiently small. The implicit function theorem implies that there exist $\delta\ll \inf_{\mu\in I}m'(\mu)$ and a unique $C^1$ map $\boldsymbol{\sigma}=\boldsymbol{\sigma}(\psi)$ such that $G(\psi,\boldsymbol{\sigma}(\psi))=0$ in a neighborhood $V_{\boldsymbol{\sigma}_0}$ of $\sum_{l=1}^k\Psi_{\sigma_{0l},\varepsilon}$.
	
	Since $\boldsymbol{\Sigma}$ is connected and $\boldsymbol{\sigma}_0$ is chosen arbitrarily in $\boldsymbol{\Sigma}$, we can take $V_{\boldsymbol{\bar{\sigma}}}\cap V_{\boldsymbol{\sigma}_0}\neq \emptyset$ and expand $V_{\boldsymbol{\sigma}_0}$ to cover the whole neighborhood $U_\delta$ of $\mathbf{M}$. This completes the proof of Proposition \ref{prop5-2}. 
\end{proof}

Now we estimate $||\psi_0-\sum_{l=1}^k\Psi_{\sigma_j(0),\varepsilon}||_{H^1}$. We have
\begin{equation}\label{5-16}
    ||\Psi_{\sigma_{j}(0),\varepsilon}-\Psi_{\sigma_{j,0},\varepsilon}||_{H^1}\le C|\sigma_j(0)-\sigma_{j,0}|.
\end{equation}
Using that $\boldsymbol{\sigma}(0)=\boldsymbol{\sigma}(\psi_0)$,  $\boldsymbol{\sigma_0}=\boldsymbol{\sigma}(\sum_{l=1}^k\Psi_{\sigma_{l,0},\varepsilon})$ and $\sup_{\psi\in U_\delta}|\partial_\psi\boldsymbol{\sigma}(\psi)|\le C$, we have
\begin{equation}\label{5-17}
	|\boldsymbol{\sigma}(0)-\boldsymbol{\sigma}_0|\le C||\psi_0-\sum_{l=1}^k\Psi_{\sigma_{j,0},\varepsilon}||_{H^1}\le C\varepsilon_v.
\end{equation}
Thus from \eqref{5-16} and \eqref{5-17} we deduce that
\begin{equation}\label{5-18}
	||\psi_0-\sum_{l=1}^k\Psi_{\sigma_j(0),\varepsilon}||_{H^1}\le \sum_{l=1}^k||\Psi_{\sigma_{j}(0),\varepsilon}-\Psi_{\sigma_{j,0},\varepsilon}||_{H^1}+||\psi_0-\sum_{l=1}^k\Psi_{\sigma_{j,0},\varepsilon}||_{H^1} \le C\varepsilon_v.
\end{equation}

We have already known that for given initial data in Section 2, the Cauchy problem \eqref{1-3} has a solution in $C(\mathbb{R}^N; H^1(\mathbb{R}^N,\mathbb{C}))\cap C(\mathbb{R}^N; H^{-1}(\mathbb{R}^N,\mathbb{C}))$. So for $0\le t\le T$, if $\psi$ stays in the neighborhood $U_\delta$ of $\mathbf{M}_k$, then the $C^1$ trajectory $\boldsymbol{\sigma}(\psi(t))$ traced out by $\boldsymbol{\sigma}$ is unique. In the rest part of this paper, we may always assume $\varepsilon_v\ll \delta$.

\section{Equation of motion}
To study the dynamics of solitons, we should split $\psi$ into $k$ peaks, which is accomplished by making some smooth truncation near each soliton. To simplify notation, let $\mathcal{T}_j=\mathcal{T}_{a_jv_j\gamma_j}$,  and define
\begin{equation}\label{6-1}
u_j:=\mathcal{T}_j^{-1}(\varphi_j\cdot\psi), \ \ \ j=1,\cdots,k-1
\end{equation}
and
\begin{equation}\label{6-2}
u_k:=\mathcal{T}_k^{-1}(\psi-\sum\limits_{l=1}^{k-1}\mathcal{T}_lu_l),
\end{equation}
where $\varphi_j\in C_c^{\infty}$ are truncating functions satisfying
\begin{equation}\label{6-3}
\varphi_j(x)=\left\{
\begin{array}{lll}
1  \ \ \ & x\in B_L(a_{j,0}), \\
0            & x\in \mathbb{R}^N\setminus B_{2L}(a_{j,0})
\end{array}
\right.
\end{equation}
with $j=1,\cdots,k-1$. According to Proposition \ref{5-2}, for $\psi\in U_\delta$, it has a skew orthogonal decomposition:  
\begin{equation}\label{6-4}
\psi=\sum\limits_{l=1}^{k}\mathcal{T}_l\eta_{\mu_l,\varepsilon}+w.
\end{equation}
Define $w_j$ as
\begin{equation}\label{6-5}
\begin{split}
w_j:&=\mathcal{T}_j^{-1}(\varphi_j\cdot(\psi-\mathcal{T}_j\eta_{\mu_j,\varepsilon}))\\
&=\mathcal{T}_j^{-1}(\varphi_j\cdot(\sum\limits_{l\neq j}^{k}\mathcal{T}_l\eta_{\mu_l,\varepsilon}+w))), \ \ \ j=1,\cdots,k-1 
\end{split}
\end{equation}
and
\begin{equation}\label{6-6}
w_k:=\mathcal{T}_k^{-1}(w-\sum\limits_{l=1}^{k-1}\mathcal{T}_lw_l).
\end{equation}
Since $|a_{i}-a_{j}|>2L$ and $|v_j|< K\varepsilon$ for $i,j=1,\cdots,k$, $i\neq j$, by the definition of $H^1$ norm \eqref{1-4}, it holds for every fixed time $t$ that
\begin{equation}\label{6-7}
(\frac{1}{2}-\sup |v_j|)\sum\limits_{l=1}^k||w_l||_{H^1}+O(e^{-\frac{c}{\varepsilon}})\le||w||_{H^1}\le (1+\sup |v_j|)\sum\limits_{l=1}^k||w_l||_{H^1}+O(e^{-\frac{c}{\varepsilon}}), 
\end{equation}
provided that $\psi\in U_\delta$. From \eqref{6-5} and \eqref{6-6}, we also know that
\begin{equation}\label{6-8}
\mathcal{T}_ju_j(x)-\mathcal{T}_jw_j(x)=\left\{
\begin{array}{lll}
\Psi_{\sigma_j,\varepsilon}+O(\sum\limits_{i\neq j}^{k}e^{-\frac{c|x-a_i|}{\varepsilon}}) \ \ \ \forall x\in B_{2L}(a_{j,0}),\\
0  \ \ \ \ \ \ \ \ \ \  \ \ \ \ \ \ \ \ \ \  \ \ \ \ \ \ \ \ \ \ \ \forall x \in \mathbb{R}^N\setminus B_{2L}(a_{j,0}),
\end{array}
\right.
\end{equation}
for $j=1,\cdots,k-1$, and
\begin{equation}\label{6-9}
\mathcal{T}_ku_k(x)-\mathcal{T}_kw_k(x)=\left\{
\begin{array}{lll}
\Psi_{\sigma_k,\varepsilon}+O(\sum\limits_{j=1}^{k-1}e^{-\frac{c|x-a_j|}{\varepsilon}}) \ \ \ \forall x\in \mathbb{R}^N\setminus\cup_{j=1}^{k-1}B_{L}(a_{j,0}),\\
0  \ \ \ \ \ \ \ \ \ \  \ \ \ \ \ \ \ \ \ \  \ \ \ \ \ \ \ \ \ \ \ \forall x \in\cup_{j=1}^{k-1}B_{L}(a_{j,0}).
\end{array}
\right.
\end{equation}

We introduce the generators for $m=1,\cdots,N$:
\begin{equation*}
\mathcal{K}_{j,m}=-\varepsilon\mathcal{T}_j\partial_m, \ \ \mathcal{K}_{j,N+m}=\frac{(x-a_j)_m}{\varepsilon}\mathbf{i}\mathcal{T}_j, \ \ \mathcal{K}_{j,2N+1}=\mathbf{i}\mathcal{T}_j, \ \ \mathcal{K}_{j,2N+2}=\mathcal{T}_j\partial_{\mu_j},
\end{equation*}
and following coefficients corresponding to the basis $z_{j,m}$ defined in \eqref{5-5}
\begin{equation}\label{6-10}
\beta_{j,m}=\frac{\dot{a}_{j,m}}{\varepsilon^2}-\frac{v_{j,m}}{\varepsilon}, \ \ \ j=1,\cdots,k, \ \ \ m=1,\cdots,N,
\end{equation}
\begin{equation}\label{6-11}
\beta_{j,N+m}=\frac{1}{2}\dot{v}_{j,m}+\varepsilon\partial_{x_m} V_\varepsilon(a_j), \ \ \ j=1,\cdots,k, \ \ \ m=1,\cdots,N,
\end{equation}
\begin{equation}\label{6-12}
\beta_{j,2N+1}=\dot{\gamma}-\frac{\dot{a}_j}{2\varepsilon}\cdot v_j-\mu_j+\frac{|v_j|^2}{4}+V_\varepsilon(a_j), \ \ \ j=1,\cdots,k,
\end{equation}
\begin{equation}\label{6-13}
\beta_{j,2N+2}=\dot{\mu}_j, \ \ \ j=1,\cdots,k,
\end{equation}
where $a_{j,m}$ and $v_{j,m}$ are $m$-th component of $a_j$ and $v_j$. Denote
\begin{equation}\label{6-14}
\beta_j\cdot \mathcal{K}_j:=\sum\limits_{m=1}^{2N+2}\beta_{j,m}\cdot \mathcal{K}_{j,m} \ \ and \ \ \tilde{\beta}_j\cdot \tilde{\mathcal{K}}_j:=\sum\limits_{m=1}^{2N+1}\beta_{j,m}\cdot \mathcal{K}_{j,m}.
\end{equation}
We have the following theorem, which reparameterize the NSE \eqref{1-3} into the form we needed.
\begin{theorem}\label{thm5-1}
	If $\psi\in U_{\delta}$ satisfies \eqref{1-3}, then it holds
	\begin{equation}\label{6-15}
	\sum\limits_{l=1}^{k}\tilde{\beta}_l\cdot \tilde{\mathcal{K}}_lu_l+\sum\limits_{l=1}^{k}\mathcal{T}_l\dot{u_l}=\sum\limits_{l=1}^kJ\mathcal{T}_l(-\varepsilon^2\Delta u_l+\mu_lu_l)
	+\sum\limits_{l=1}^kJ\mathcal{R}_l\cdot\mathcal{T}_lu_l-Jf(\psi).
	\end{equation}
	where 
	\begin{equation}\label{6-16}
	\mathcal{R}_l=V_\varepsilon(x)-V_\varepsilon(a_l)-\nabla V_\varepsilon(a_l)\cdot (x-a_l)=O((\varepsilon_vx)^2),
	\end{equation}
	and $J$ is the operator of multiplication by $\mathbf{i}^{-1}$ defined in \eqref{3-3}.
\end{theorem}
\begin{proof}
	We can rewrite the first equation of \eqref{1-3}: 
	$$\mathbf{i}\partial_t\psi=-\varepsilon^2\Delta\psi+V_\varepsilon(x)\psi-f(\psi),$$
	as 
	\begin{equation}\label{6-17}
   	\sum\limits_{l=1}^{k}\partial_t(\mathcal{T}_lu_l)=\sum\limits_{l=1}^{k}J(-\varepsilon^2\Delta \mathcal{T}_lu_l+V_\varepsilon(x)\mathcal{T}_lu_l)-Jf(\psi).
	\end{equation}
	Since $\mathcal{T}_jv=e^{\frac{\mathbf{i}}{2}v_j\cdot\frac{x-a_j}{\varepsilon}+\gamma_j}v(x-a_j)$, for the left side of \eqref{6-17}, we have
	\begin{equation}\label{6-18}
	\begin{split}
	\sum\limits_{l=1}^{k}\partial_t(\mathcal{T}_lu_l)&=	\sum\limits_{l=1}^{k}(\frac{1}{2}\dot{v}\cdot\frac{x-a_l}{\varepsilon}-\frac{1}{2}v\cdot\frac{\dot{a}}{\varepsilon}+\dot{\gamma})\mathbf{i}\mathcal{T}_lu_l\\
	&+\sum\limits_{l=1}^{k}\frac{\dot{a_l}}{\varepsilon}(-\mathcal{T}_l\nabla u_l)+\sum\limits_{l=1}^{k}\mathcal{T}_l\dot{u_l},
	\end{split}
	\end{equation}
	While the right side of \eqref{6-17} can be rewritten as
	\begin{equation}\label{6-19}
	\begin{split}
	\sum\limits_{l=1}^{k}J(&-\varepsilon^2\Delta \mathcal{T}_lu_l+V_\varepsilon(x)\mathcal{T}_lu_l)-Jf(\psi)\\
	&=\sum\limits_{l=1}^k\big(J\mathcal{T}_l(-\varepsilon^2\Delta u_l)+v_l(-\mathcal{T}_l\nabla u_l)-\frac{v_l^2}{4}(\mathbf{i}\mathcal{T}_lu_l)+JV_\varepsilon(x)(\mathcal{T}_lu_l)\big)-Jf(\psi).
	\end{split}
	\end{equation}
	Adding 
	\begin{equation}\label{6-20}
	\sum\limits_{l=1}^k\big(\mu_l(-\mathbf{i}\mathcal{T}_lu_l)-JV_\varepsilon(a_l)\mathcal{T}_lu_l-J\nabla V_\varepsilon(a_l)(x-a_l)\mathcal{T}_lu_l\big)
	\end{equation}
	and using notations \eqref{6-10}-\eqref{6-14} to collect parameters into $\tilde{\beta}$, we conclude \eqref{6-15}.	
\end{proof}

We want to derive the modulation formula for solitons dynamics. To split the nonlinear term $f(\psi)$, using the definition of $u_j$ and $w_j$, we deduce that
\begin{equation}\label{6-21}
f(\psi)=\sum\limits_{l=1}^k \mathcal{T}_j f(u_j)+\zeta(w)+O(e^{-\frac{c}{\varepsilon}}),
\end{equation}
where $\zeta(w)\in H^1_0(\mathbb{R}^N)$ is the boundary term with $supp\{\zeta(w)\}\subset\cup_{l=1}^{k-1}\big(B_{2L}(a_{l,0})\setminus B_L(a_{l,0})\big)$. Thus for the right side of \eqref{6-15}, we have
\begin{equation*}
\begin{split}
\sum\limits_{l=1}^kJ\mathcal{T}_l(-\varepsilon^2\Delta u_l+\mu_lu_l)
&+\sum\limits_{l=1}^kJ\mathcal{R}_l\cdot\mathcal{T}_lu_l-Jf(\psi)\\
&=\sum\limits_{l=1}^kJ\mathcal{E}'_{\mu_l,\varepsilon}(u_l)+\sum\limits_{l=1}^kJ\mathcal{R}_l\cdot\mathcal{T}_lu_l-J\zeta(w)+O(e^{-\frac{c}{\varepsilon}}).
\end{split}
\end{equation*}
From \eqref{6-8}, we obtain
\begin{equation}\label{6-22}
\begin{split}
\mathcal{E}'_{\mu_j,\varepsilon}(u_j)&=\mathcal{E}'_{\mu_j,\varepsilon}(\eta_{\mu_j,\varepsilon}+w_j)+O(e^{-\frac{c}{\varepsilon}})\\
&=\mathcal{L}_{\eta_{\mu_j,\varepsilon}}(w_j)+N_{\eta_{\mu_j,\varepsilon}}(w_j)+O(e^{-\frac{c}{\varepsilon}})
\end{split}
\end{equation}
Since $\dot\eta_{\mu_j,\varepsilon}=\dot{\mu}_j\partial_{\mu_j}\eta_{\mu_j,\varepsilon}$, we can rewrite \eqref{6-15} as
\begin{equation}\label{6-23}
\begin{split}
\sum\limits_{l=1}^k\big(\beta_l\cdot\mathcal{K}\eta_{\mu_l,\varepsilon}+\tilde{\beta}_l\cdot\tilde{\mathcal{K}}w_l+\mathcal{T}_l\dot{w_l}\big)=\sum\limits_{l=1}^k\big(J\mathcal{T}_l\mathcal{L}_{\eta_{\mu_l,\varepsilon}}(w_l)&+J\mathcal{T}_lN_{\eta_{\mu_l,\varepsilon}}(w_l)+JR_l\cdot \mathcal{T}_lu_l)\big)\\
&+J\zeta(w)+O(e^{-\frac{c}{\varepsilon}}).
\end{split}
\end{equation}
If we denote
\begin{equation}\label{6-24}
W(w):=\sum\limits_{l=1}^k\big(J\mathcal{T}_l\mathcal{L}_{\eta_{\mu_l,\varepsilon}}(w_l)+J\mathcal{T}_lN_{\eta_{\mu_l,\varepsilon}}(w_l)+JR_l\cdot \mathcal{T}_lw_l-\tilde{\beta}_l\cdot\tilde{\mathcal{K}}w_l\big)+J\zeta(w),
\end{equation}
and
\begin{equation}\label{6-25}
Q(\boldsymbol{\sigma}):=\sum\limits_{l=1}^k\big(JR_l\cdot \mathcal{T}_l\eta_{\mu_l,\varepsilon}-J\mathcal{T}_l\mathcal{L}_{\eta_{\mu_l,\varepsilon}}(\eta_{\mu_l,\varepsilon})\big),
\end{equation}
we obtain
\begin{equation}\label{6-26}
\sum\limits_{l=1}^k\mathcal{T}_l\dot{w_l}=W(w)+Q(\boldsymbol{\sigma})+O(e^{-\frac{c}{\varepsilon}}).
\end{equation}

We will let $Jz_j$ act on \eqref{6-23}, and hope that the surplus variables can be eliminated by the skew orthogonal condition. For this purpose, we give some useful facts for $z_j\in \mathbf{T}_{\Psi_{\sigma_j,\varepsilon}}\mathbf{M}$ with $j=1,\cdots,k$ . Using the definition of decomposition in Proposition \ref{prop5-2}, we have $\langle Jz_j,J\mathcal{L}_{\eta_{\mu_j,\varepsilon}}(w)\rangle_\varepsilon=0$. By the smooth truncation as we have done in \eqref{6-1} and \eqref{6-5}, it holds
\begin{equation}\label{6-27}
\langle Jz_j,J\mathcal{T}_j\mathcal{L}_{\eta_{\mu_j,\varepsilon}}(w_j)\rangle_\varepsilon=O(e^{-\frac{c}{\varepsilon}}).
\end{equation}
Similarly, for $z_i\in \mathbf{T}_{\Psi_{\sigma_i,\varepsilon}}\mathbf{M}$ and $i\neq j$, we have
\begin{equation}\label{6-28}
\langle Jz_i,J\mathcal{T}_j\mathcal{L}_{\eta_{\mu_j,\varepsilon}(w_j)}\rangle_\varepsilon=O(e^{-\frac{c}{\varepsilon}}).
\end{equation}
Then, we consider the derivatives with respect to time $t$. For $\bar{z}_j\in \mathbf{T}_{\Psi_{\sigma_{0,j},\varepsilon}}\mathbf{M}$ where $\sigma_{0,j}=(0,0,0,\mu_j)$, and $j=1,\cdots,k$, it holds
\begin{equation}\label{6-29}
\begin{split}
0=\partial_t\langle J\bar{z}_j,\mathcal{T}_j^{-1}w\rangle_\varepsilon&=\langle J\bar{z}_j,\partial_t(\mathcal{T}_j^{-1}w)\rangle_\varepsilon+\dot\mu_j\langle J\partial_{\mu_j}\bar{z}_j,\mathcal{T}_j^{-1}w\rangle_\varepsilon\\
&=\langle Jz_j,\mathcal{T}_j\partial_t(\mathcal{T}_j^{-1}w)\rangle_\varepsilon+\dot\mu_j\langle J\partial_{\mu_j}z_j,w\rangle_\varepsilon,
\end{split}
\end{equation}
where $z_j=\mathcal{T}_j\bar{z}_j\in\mathbf{T}_{\Psi_{\sigma_j,\varepsilon}}\mathbf{M}$. We can use \eqref{6-3} again to deduce
\begin{equation}\label{6-30}
\begin{split}
\langle Jz_j,\mathcal{T}_j\dot{w}_j\rangle_\varepsilon&=-\dot{\mu}_j\langle J\partial_{\mu_j}z_j,\mathcal{T}_jw_j\rangle_\varepsilon+O(e^{-\frac{c}{\varepsilon}})\\
&=\beta_{j,2N+2}\langle Jz_j,\mathcal{K}_{j,2N+2}w_j\rangle_\varepsilon+O(e^{-\frac{c}{\varepsilon}}).
\end{split}
\end{equation}
While for $z_i\in \mathbf{T}_{\Psi_{\sigma_i,\varepsilon}}\mathbf{M}$ and $i\neq j$ it holds
\begin{equation}\label{6-31}
\langle Jz_i,\mathcal{T}_j\dot{w}_j\rangle_\varepsilon=O(e^{-\frac{c}{\varepsilon}}).
\end{equation}

Now, we let $Jz_{i,n}$ act on both sides of \eqref{6-23}. Since the support of $J\zeta(w)$ is contained in $\cup_{l=1}^{k-1}\big(B_{2L}(a_{l,0})\setminus B_L(a_{l,0})\big)$, we have $\langle Jz_{i,n},J\zeta(w)\rangle_\varepsilon=O(e^{-\frac{c}{\varepsilon}})$ for $i=1,\cdots,k$ and $n=1,\cdots,2N+2$.
Recalling that $\mathcal{K}_{j,m}\eta_{\mu_j,\varepsilon}=z_{j,m}$ are exactly the skew orthogonal bases of $\mathbf{T}_{\Psi_{\sigma_j,\varepsilon}}\mathbf{M}$, we can obtain
\begin{equation}\label{6-32}
\begin{split}
\sum\limits_{j=1}^k\sum\limits_{m=1}^{2N+2}\Omega_{j,m}^{i,n}\beta_{j,m}=\langle z_{i,n}, \sum\limits_{l=1}^k\big(\mathcal{T}_lN_{\eta_{\mu_l,\varepsilon}}&(w_l)+R_l\cdot \mathcal{T}_l(\eta_{\mu_l,\varepsilon}+w_l)\big)\rangle_\varepsilon\\
&-\sum\limits_{l=1}^k \beta_l\langle J z_{i,n}, \mathcal{K}_lw_l\rangle_\varepsilon+O(e^{-\frac{c}{\varepsilon}}),
\end{split}
\end{equation}
where
\begin{equation}\label{6-33}
\Omega_{j,m}^{i,n}=\langle z_{i,n}, J^{-1}z_{j,m}\rangle_\varepsilon=\left\{
\begin{array}{lll}
\Omega_{\eta_{\mu_i,\varepsilon}}^{m,n} \ \ \ & i=j, \\
O(e^{-\frac{c}{\varepsilon}})        & i\neq j.
\end{array}
\right.
\end{equation}
Thus, using the notations \eqref{6-10}-\eqref{6-13} and the expression \eqref{4-15} of matrix $	\Omega_{\eta_{\mu,\varepsilon}}$, we can rewrite \eqref{6-32} as four equations:
\begin{equation}\label{6-34}
\begin{split}
\frac{\dot{a}_{j,m}}{\varepsilon^2}=\frac{v_{j,m}}{\varepsilon}+(m(\mu_j))^{-1}\big(\langle z_{j,N+m},(\mathcal{T}_jN_{\eta_{\mu_j,\varepsilon}}&(w_j)+R_j\cdot \mathcal{T}_j(w_j))\ \rangle_\varepsilon\\
&-\beta_j\cdot\langle Jz_{j,N+m}, \mathcal{K}_jw_j\rangle_\varepsilon\big)+O(e^{-\frac{c}{\varepsilon}}),
\end{split}
\end{equation}
\begin{equation}\label{6-35}
\begin{split}
\frac{1}{2}\dot{v}_{j,m}=-\varepsilon\partial_{x_m}V_\varepsilon(a_j)&-(m(\mu_j))^{-1}\big(\langle z_{j,m},(\mathcal{T}_jN_{\eta_{\mu_j,\varepsilon}}(w_j)+R_j\cdot \mathcal{T}_j(w_j))\ \rangle_\varepsilon\\
&+\langle z_{j,m}, R_j\cdot \mathcal{T}_j(\eta_{\mu_j,\varepsilon})\rangle_\varepsilon-\beta_j\cdot\langle Jz_{j,m}, \mathcal{K}_jw_j\rangle_\varepsilon\big)+O(e^{-\frac{c}{\varepsilon}}),
\end{split}
\end{equation}
\begin{equation}\label{6-36}
\begin{split}
\dot{\gamma}_j=\mu_j+\frac{v_j}{2}\cdot\frac{\dot{a_j}}{\varepsilon}&-\frac{|v_j|^2}{4}-V_\varepsilon(a_j)
\\&-(m'(\mu_j))^{-1}\big(\langle z_{j,2N+2},(\mathcal{T}_jN_{\eta_{\mu_j,\varepsilon}}(w_j)+R_j\cdot \mathcal{T}_j(w_j))\ \rangle_\varepsilon
\\&+\langle z_{j,2N+2}, R_j\cdot \mathcal{T}_j(\eta_{\mu_j,\varepsilon})\rangle_\varepsilon-\beta_j\cdot\langle Jz_{j,2N+2}, \mathcal{K}_jw_j\rangle_\varepsilon\big)+O(e^{-\frac{c}{\varepsilon}}),
\end{split}
\end{equation}
\begin{equation}\label{6-37}
\begin{split}
\dot{\mu}_j=(m'(\mu_j))^{-1}\big(\langle z_{j,2N+1},(\mathcal{T}_jN_{\eta_{\mu_j,\varepsilon}}(w_j)&+R_j\cdot \mathcal{T}_j(w_j))\ \rangle_\varepsilon\\
&-\beta_j\cdot\langle Jz_{j,2N+1}, \mathcal{K}_jw_j\rangle_\varepsilon\big)+O(e^{-\frac{c}{\varepsilon}}),
\end{split}
\end{equation}
where $m(\mu_j)$, $m'(\mu_j)$ are given in \eqref{4-14} \eqref{4-16}, and we use the facts $\langle J\eta_{\mu_j,\varepsilon},R_j\eta_{\mu_j,\varepsilon}\rangle_\varepsilon=0$, $\langle \frac{x_m}{\varepsilon}J\eta_{\mu_j,\varepsilon},R_j\eta_{\mu_j,\varepsilon}\rangle_\varepsilon=0$, $\langle \frac{x_m}{\varepsilon}\eta_{\mu_j,\varepsilon},\eta_{\mu_j,\varepsilon}\rangle_\varepsilon=0$ for $j=1,\cdots,k$ and $m=1,\cdots,N$.

If we denote $\underline{\sigma}_j=(\frac{a_j}{\varepsilon^2}, \frac{v_j}{2}, \gamma_j, \mu_j)$ and $\underline{\boldsymbol{\sigma}}=(\underline{\sigma}_1,\cdots,\underline{\sigma}_1)$, we can abbreviate \eqref{6-34}-\eqref{6-37} as
\begin{equation}\label{6-38}
\underline{\dot{\boldsymbol{\sigma}}}=X(\underline{\boldsymbol{\sigma}})-X_\delta(\underline{\boldsymbol{\sigma}},w)+O(e^{-\frac{c}{\varepsilon}}),
\end{equation}
where
\begin{equation}\label{6-39}
X_{j,m}(\underline{\boldsymbol{\sigma}})=\underline{\dot\sigma}_{j,m}-\beta_{j,m}
\end{equation}
and
\begin{equation}\label{6-40}
\begin{split}
X_{\delta}^{i,n}(\underline{\boldsymbol{\sigma}},w)=\sum\limits_{j=1}^k\sum\limits_{m=1}^{2N+2}(\Omega_{j,m}^{i,n})^{-1}\big(\langle z_{i,n}, \sum\limits_{l=1}^k\big(\mathcal{T}_lN_{\eta_{\mu_l,\varepsilon}}(w_l)&+R_l\cdot \mathcal{T}_l(\eta_{\mu_l,\varepsilon}+w_l)\big)\rangle_\varepsilon
\\
&-\sum\limits_{l=1}^k \beta_l\langle Jz_{i,n}, \mathcal{K}_jw_l\rangle_\varepsilon\big).
\end{split}
\end{equation}
Notice that the last term on the right side of \eqref{6-40} is of order $O\big(|\beta|\sum\limits_{l=1}^k||w_l||_{L^2}\big)$ by the Green's theorem. On the other hand, it holds $R_j\le C\varepsilon_v^2$ and $N_{\eta_{\mu_j,\varepsilon}}(w_j)\le C||w_j||_{H^1}^2$ for $j=1,\cdots,k$. By the left side of \eqref{6-7}, we have
\begin{equation}\label{6-41}
\begin{split}
X_{\delta}(\underline{\boldsymbol{\sigma}},w)&=O(|\beta|_\infty\sum\limits_{l=1}^k||w_l||_{L^2}+\varepsilon_v^2+\sum\limits_{l=1}^k||w_l||^2_{H^1})\\
&=O(|\beta|_\infty||w||_{L^2}+\varepsilon_v^2+||w||^2_{H^1})
\end{split}
\end{equation}
with $|\beta|_\infty=\max|\beta_{j,m}|$, provided $||w||_{H^1}\le \delta$. 

To summarize, we get the following proposition:
\begin{proposition}\label{prop6-1}
	$\underline{\boldsymbol{\sigma}}$ and perturbation $w_j$ satisfy
		\begin{equation}\label{6-42}
		\underline{\dot{\boldsymbol{\sigma}}}=X(\underline{\boldsymbol{\sigma}})-X_\delta(\underline{\boldsymbol{\sigma}},w)+O(e^{-\frac{c}{\varepsilon}})
		\end{equation}
		and
		\begin{equation}\label{6-43}
		\sum\limits_{l=1}^k\mathcal{T}_l\dot{w_l}=W(w)+Q(\boldsymbol{\sigma})+O(e^{-\frac{c}{\varepsilon}}),
		\end{equation}
		where $W(w)$ and $Q(\boldsymbol{\sigma})$ are defined in \eqref{6-24} and \eqref{6-25}. Furthermore, we have the following estimate for the vector field $X_\delta$:
		\begin{equation}\label{6-44}
		X_\delta=O(|\beta|_\infty||w||_{L^2}+\varepsilon_v^2+||w||^2_{H^1})
		\end{equation}
		with $|\beta|_\infty=\max|\beta_{j,m}|$, provided $||w||_{H^1}\le \delta$.
\end{proposition}

In the next two sections, we are going to show $\sup_{t\in(0,T]}||w_l||_{H^1}=O(\varepsilon_v)$ for $T=\frac{T_0}{\varepsilon^2}$ with $T_0$ a positive constant, from which $\sup_{t\in(0,T]}|X_\delta|=O(\varepsilon_v^2)$ and the main result follows.

\section{Evolution of energy}

By the definition of $u_j$, we see that $u_j$ is very close to $\eta_{\mu_j,\varepsilon}$. In this section, we are going to prove that $\sum\limits_{l=1}^k\partial_t\mathcal{E}_{\mu_l,\varepsilon}(u_l)-\sum\limits_{l=1}^k\partial_t\mathcal{E}_{\mu_l}(\eta_{\mu_l,\varepsilon})$ is small, so that we are able to control $\sum\limits_{l=1}^k||w_l||_{H^1}$ or $||w||_{H^1}$ in a large time interval. First, we have following lemma, which gives the derivative with respect to $t$ of energy functional $\sum\limits_{l=1}^k\mathcal{E}_{\mu_l,\varepsilon}(u_l)$:
\begin{lemma}\label{lem7-1}
	It holds
	\begin{equation}\label{7-1}
	\begin{split}
	\sum\limits_{l=1}^k\partial_t\mathcal{E}_{\mu_l,\varepsilon}(u_l)=\frac{1}{2}\sum\limits_{l=1}^k\dot{\mu}_l||u_l||_{L^2}^2&-\sum\limits_{l=1}^k\langle(\frac{\dot v_l}{2}+\varepsilon\nabla V_{\varepsilon,a_l})\mathbf{i}u_l,\varepsilon\nabla u_l\rangle_\varepsilon\\
	&+C\varepsilon^2||w||_{H^1}^2+O(e^{-\frac{c}{\varepsilon}}),
	\end{split}
	\end{equation}
	where $V_{\varepsilon,a}(x):=V_\varepsilon(x+a)$ for any $a\in \mathbb{R}^N$.
\end{lemma}
\begin{proof}
	According to the definition of $\mathcal{H}_{v,\varepsilon}$ in \eqref{3-4}, and the definition of $\mathcal{E}_{\mu_l,\varepsilon}(\eta_{\mu,\varepsilon})$ in \eqref{3-5}, we have
	\begin{equation}\label{7-2}
	\begin{split}
	2\mathcal{H}_{v,\varepsilon}&(\psi)=2\mathcal{H}_{v,\varepsilon}(\sum\limits_{l=1}^k\mathcal{T}_lu_l)\\
	&=2\sum\limits_{l=1}^k\mathcal{E}_{\mu_l,\varepsilon}(u_l)-\sum\limits_{l=1}^k(\frac{|v_l|^2}{4}+\mu_l)||\mathcal{T}_lu_l||_{L^2}+\sum\limits_{l=1}^kv_l\cdot\langle\mathbf{i}\mathcal{T}_lu_l, \varepsilon\nabla(\mathcal{T}_lu_l)\rangle_\varepsilon+\int_{\mathbb{R}^N}V_\varepsilon|\psi|^2dx
	\end{split}
	\end{equation}
    We caculate the derivative with respect to $t$ of each individual term in \eqref{7-2}. Since $\partial_t\big(||\psi||_{L^2}^2\big)=0$ and $|v_j|< K\varepsilon$ for $j=1,\cdots,k$, we use condition (C) in Section 2 and \eqref{3-9} to deduce that
    \begin{equation}\label{7-3}
    \begin{split}
    \sum\limits_{l=1}^k(\frac{|v_l|^2}{4}+\mu_l)\partial_t\big(||\mathcal{T}_lu_l||_{L^2}^2\big)&\le \sum\limits_{l=1}^k(\frac{(k\varepsilon)^2}{4}+C)\partial_t\big(||\mathcal{T}_lu_l||_{L^2}^2\big)\\
    &= C\varepsilon\mathbf{Re}\sum\limits_{l=1}^{k-1}\int_{B_{2L}(a_l(0))\setminus B_L(a_l(0))}\frac{\mathbf{i}}{\varepsilon^{N-1}}w\cdot\nabla w\cdot\nabla(\varphi_l^2)dx+O(e^{-\frac{c}{\varepsilon}})\\
    &\le C\varepsilon\sum\limits_{l=1}^{k-1}\int_{B_{2L}(a_l(0))\setminus B_L(a_l(0))}\frac{1}{\varepsilon^{N-1}}|w|^2|\nabla^2(\varphi_l^2)|dx+O(e^{-\frac{c}{\varepsilon}})\\
    &\le C\varepsilon^2||w||_{H^1}^2+O(e^{-\frac{c}{\varepsilon}}),
    \end{split}
    \end{equation}
    where \eqref{3-9} and the smooth truncation we have made in \eqref{6-1} have been used in obtaining the first equality. Similarly, by \eqref{3-10} we have
    \begin{equation}\label{7-4}
    \begin{split}
    \partial_t\langle\mathbf{i}\mathcal{T}_ju_j, \varepsilon\nabla(\mathcal{T}_ju_j)&\rangle_\varepsilon-\langle\varepsilon(\nabla V_\varepsilon)\mathcal{T}_ju_j,\mathcal{T}_ju_j\rangle_\varepsilon\\
    &\le C\varepsilon\int_{B_{2L}(a_j(0))\setminus B_L(a_j(0))}\frac{1}{\varepsilon^N}(|\varepsilon\nabla w|^2+|w|^2+|w|^p)||\varphi_j||_{C^3}^2dx\\
    &\ \ \ +C\varepsilon\int_{B_{2L}(a_j(0))\setminus B_L(a_j(0))}\frac{1}{\varepsilon^{N-2}}(|w|^2+|w||\nabla w|)||\varphi_j||_{C^3}^2dx+O(e^{-\frac{c}{\varepsilon}})\\
    &=C\varepsilon||w||_{H^1}^2+O(e^{-\frac{c}{\varepsilon}}), \ \ \ \ \ \ \ \ \ \ \ \ \ \ \  j\neq k,\\
    \partial_t\langle\mathbf{i}\mathcal{T}_ku_k, \varepsilon\nabla(\mathcal{T}_ku_k)&\rangle_\varepsilon-\langle\varepsilon(\nabla V_\varepsilon)\mathcal{T}_ku_k,\mathcal{T}_ju_k\rangle_\varepsilon\\
    &\le C\varepsilon\sum\limits_{l=1}^{k-1}\int_{B_{2L}(a_l(0))\setminus B_L(a_l(0))}\frac{1}{\varepsilon^N}(|\varepsilon\nabla w|^2+|w|^2+|w|^p)||\varphi_l||_{C^3}^2dx\\
    & \ \ \ +C\varepsilon\sum\limits_{l=1}^{k-1}\int_{B_{2L}(a_l(0))\setminus B_L(a_l(0))}\frac{1}{\varepsilon^{N-2}}(|w|^2+|w||\nabla w|)||\varphi_j||_{C^3}^2dx+O(e^{-\frac{c}{\varepsilon}})\\
    &=C\varepsilon||w||_{H^1}^2+O(e^{-\frac{c}{\varepsilon}}),\\
    \end{split}
    \end{equation}
    provided $||w||_{H^1}\le \delta$. Using \eqref{3-8}, \eqref{3-10}, \eqref{7-3}, \eqref{7-4} and $\partial_t\mathcal{H}_{V,\varepsilon}(\psi)=0$, and taking derivative with respect to $t$ of \eqref{7-2}, we obtain
    \begin{equation}\label{7-5}
    \begin{split}
    2\sum\limits_{l=1}^k\partial_t\mathcal{E}_{\mu_l,\varepsilon}(u_l)&=\sum\limits_{l=1}^k(\dot{v}_l\cdot\frac{v_l}{2}+\dot{\mu}_l)||\mathcal{T}_lu_l||_{L^2}^2-\sum\limits_{l=1}^k\dot{v}_l\cdot\langle\mathbf{i}\mathcal{T}_lu_l, \varepsilon\nabla(\mathcal{T}_lu_l)\rangle_\varepsilon-2\langle\varepsilon(\nabla V_\varepsilon)\mathbf{i}\psi,\varepsilon\nabla\psi\rangle_\varepsilon\\
    &+\sum\limits_{l=1}^kv_l\cdot\langle\varepsilon(\nabla V_\varepsilon)\mathcal{T}_lu_l,\mathcal{T}_lu_l\rangle_\varepsilon+C\varepsilon^2||w||_{H^1}^2,
    \end{split}
    \end{equation}
    where we have used the assumption $|v_l|< K\varepsilon$ for $l=1,\cdots,k$. By the definition of $\mathcal{T}_j$ and the smooth truncation \eqref{6-3}, we have
    $$2\langle\varepsilon(\nabla V_\varepsilon)\mathbf{i}\psi,\varepsilon\nabla\psi\rangle_\varepsilon=\sum\limits_{l=1}^k\langle\varepsilon(\nabla V_\varepsilon)\mathbf{i}\mathcal{T}_lu_l,\varepsilon\nabla(\mathcal{T}_lu_l)\rangle_\varepsilon+C\varepsilon\cdot\varepsilon_v\sum\limits_{l=1}^k||w_l||_{H^1}^2+O(e^{-\frac{c}{\varepsilon}}).$$
    Notice that according to the left side of \eqref{6-7}, we have
    $$ \sum\limits_{l=1}^k||w_l||_{H^1}^2\le C||w||_{H^1}^2,$$
    with $C$ a positive constant, and since $\varepsilon_v=\varepsilon^h$ with $h>2$, we can collect terms of the form $\dot{v_j}+2\varepsilon\nabla V_{\varepsilon,a_j}$ to give
    \begin{equation*}
    \begin{split}
    2\sum\limits_{l=1}^k\partial_t\mathcal{E}_{\mu_l,\varepsilon}(u_l)=\sum\limits_{l=1}^k\dot{\mu}_l||u_l||_{L^2}^2&-\sum\limits_{l=1}^k\langle(\dot v_l+2\varepsilon\nabla V_{\varepsilon,a_l})\mathbf{i}u_l,\varepsilon\nabla u_l\rangle_\varepsilon\\
    &+C\varepsilon^2||w||_{H^1}^2+O(e^{-\frac{c}{\varepsilon}}),
    \end{split}
    \end{equation*}
    which is the desired result.
\end{proof}

By now we can state the main result of this section, which claims that the derivative with respect to $t$ of energy difference $\sum\limits_{l=1}^k\mathcal{E}_{\mu_l,\varepsilon}(u_l)-\sum\limits_{l=1}^k\mathcal{E}_{\mu_l}(\eta_{\mu_l,\varepsilon})$ is small and can be bounded by $|\beta|_\infty$ and the $H^1$ norm of $w$.
\begin{proposition}\label{prop7-2}
	It holds
	\begin{equation}\label{7-6}
	\sum\limits_{l=1}^k\partial_t\mathcal{E}_{\mu_l,\varepsilon}(u_l)-\sum\limits_{l=1}^k\partial_t\mathcal{E}_{\mu_l}(\eta_{\mu_l,\varepsilon})=O\big(|\beta|_\infty||w||_{H^1}^2+\varepsilon^2||w||_{H^1}^2+\varepsilon\cdot\varepsilon_v^2||w||_{H^1}+e^{-\frac{c}{\varepsilon}}\big).
	\end{equation}
\end{proposition}
\begin{proof}
	Since $\eta_{\mu_j,\varepsilon}$ is a critical point of $\mathcal{E}_{\mu_j,\varepsilon}$, we have
	\begin{equation}\label{7-7}
	\partial_t\mathcal{E}(\eta_{\mu_j,\varepsilon})=\frac{1}{2}\dot{\mu}_j||\eta_{\mu_j,\varepsilon}||^2_{L^2}.
	\end{equation}
	From Lemma \ref{lem7-1}, we can devide $\sum\limits_{l=1}^k\partial_t\mathcal{E}_{\mu_l,\varepsilon}(u_l)-\sum\limits_{l=1}^k\partial_t\mathcal{E}_{\mu_l}(\eta_{\mu_l,\varepsilon})$ into the following three parts:
	\begin{equation}\label{7-8}
	I_1=\frac{1}{2}\sum\limits_{l=1}^k\dot{\mu}_l(||u_l||_{L^2}^2-||\eta_{\mu_l,\varepsilon}||_{L^2}^2),
	\end{equation}
	\begin{equation}\label{7-9}
	I_2=-\sum\limits_{l=1}^k\langle(\frac{\dot v_l}{2}+\varepsilon\nabla V_{\varepsilon,a_l})\mathbf{i}u_l,\varepsilon\nabla u_l\rangle_\varepsilon,
	\end{equation}
	\begin{equation}\label{7-10}
	I_3=C\varepsilon^2\sum\limits_{l=1}^k||w_l||_{H^1}^2+O(e^{-\frac{c}{\varepsilon}}).
	\end{equation}
	Using the decomposition $\psi=\sum\limits_{l=1}^{k}\mathcal{T}_l\eta_{\mu_l,\varepsilon}+w$ to find $0=\langle\mathcal{T}_j\eta_{\mu_j,\varepsilon}, w\rangle_\varepsilon=\langle\eta_{\mu_j,\varepsilon}, \mathcal{T}_j^{-1}w\rangle_\varepsilon$, and by the smooth truncation \eqref{6-1}, we have
	\begin{equation}\label{7-11}
	\langle \eta_{\mu_j,\varepsilon}, w_j\rangle_\varepsilon=O(e^{-\frac{c}{\varepsilon}}).
	\end{equation}
	Recall $\dot{\mu}_j=\beta_{j,2N+2}$. Hence for the first part $I_1$, it holds
	\begin{equation}\label{7-12}
	I_1=\frac{1}{2}\sum\limits_{l=1}^k\dot{\mu}_l||w_l||_{L^2}^2+O(e^{-\frac{c}{\varepsilon}})=O(|\beta|_\infty\sum\limits_{l=1}^k||w_l||_{H^1}^2+e^{-\frac{c}{\varepsilon}}).
	\end{equation}
	Similar to \eqref{7-11}, we can prove
	\begin{equation}\label{7-13}
	\langle\mathbf{i}\eta_{\mu_j,\varepsilon}, \varepsilon\nabla w_j\rangle_\varepsilon=O(e^{-\frac{c}{\varepsilon}})
	\end{equation}
	and
	\begin{equation}\label{7-14}
	\langle\mathbf{i}\varepsilon\nabla w_j, \eta_{\mu_j,\varepsilon}\rangle_\varepsilon=O(e^{-\frac{c}{\varepsilon}}).
	\end{equation}
	Notice that $\nabla V_{\varepsilon,a}-\nabla V_\varepsilon(a)=O(\varepsilon_v^2|x|)$ and $|x|\eta_{\mu_j,\varepsilon},\varepsilon|x|\nabla\eta_{\mu_j,\varepsilon}\in L^2$ for $j=1,\cdots,k$. Combining \eqref{7-13} \eqref{7-14} with $\frac{\dot v_{j,m}}{2}+\varepsilon\partial_{x_m}V_{\varepsilon,a_j}=\beta_{j,N+m}$, we can estimate the second part $I_2$ as follows:
	\begin{equation}\label{7-15}
	\begin{split}
	I_2&=-\sum\limits_{l=1}^k\langle(\frac{\dot v_l}{2}+\varepsilon\nabla V_{\varepsilon,a_l})\mathbf{i}w_l,\varepsilon\nabla w_l\rangle_\varepsilon-\sum\limits_{l=1}^k\langle\varepsilon(\nabla V_{\varepsilon,a_l})\mathbf{i}\eta_{\mu_l,\varepsilon},\varepsilon\nabla w_l\rangle_\varepsilon\\
	& \ \ \ -\sum\limits_{l=1}^k\langle\varepsilon(\nabla V_{\varepsilon,a_l})\mathbf{i}\varepsilon\nabla w_l, \eta_{\mu_l,\varepsilon}\rangle_\varepsilon+O(e^{-\frac{c}{\varepsilon}})\\
	&=-\sum\limits_{l=1}^k\langle(\frac{\dot v_l}{2}+\varepsilon\nabla V_\varepsilon(a_l))\mathbf{i}w_l,\varepsilon\nabla w_l\rangle_\varepsilon-\sum\limits_{l=1}^k\langle\varepsilon(\nabla V_{\varepsilon,a_l}-\nabla V_\varepsilon(a_l))\mathrm{i}w_l,\varepsilon\nabla w_l\rangle_\varepsilon\\
	& \ \ \ -\sum\limits_{l=1}^k\langle\varepsilon(\nabla V_{\varepsilon,a_l}-\nabla V_\varepsilon(a_l))\mathbf{i}\eta_{\mu_l,\varepsilon}, \varepsilon\nabla w_l\rangle_\varepsilon-\sum\limits_{l=1}^k\langle\varepsilon(\nabla V_{\varepsilon,a_l}-\nabla V_\varepsilon(a_l))\mathbf{i}\varepsilon\nabla w_l, \eta_{\mu_l,\varepsilon}\rangle_\varepsilon+O(e^{-\frac{c}{\varepsilon}})\\
	&=O\big(|\beta|_\infty\sum\limits_{l=1}^k||w_l||_{H^1}^2+\varepsilon\cdot\varepsilon_v\sum\limits_{l=1}^k||w_l||_{H^1}^2+\varepsilon\cdot\varepsilon_v^2\sum\limits_{l=1}^k||w_l||_{H^1}+e^{-\frac{c}{\varepsilon}}\big)
	\end{split}
	\end{equation}
   From condition (A) in Section 2, we have $\varepsilon_v=\varepsilon^h$ for some $h>2$. Thus we complete our proof by adding up \eqref{7-10} \eqref{7-12} \eqref{7-15}, and using the fact that
   $$ \sum\limits_{l=1}^k||w_l||_{H^1}^2\le C||w||_{H^1}^2$$
   for some positive constant $C$ from \eqref{6-7}.
\end{proof}

\section{Proof of the main results}

We start with a well-known lemma about coercivity of $\mathcal{L}_{\eta_{\mu,\varepsilon}}$, whose proof can be found in \cite{Fr3} and \cite{We2}. It relies on the fact that $\mathcal{L}_{\eta_{\mu,\varepsilon}}$ has only one nagetive eigenvalue and the structure of its null space.
\begin{lemma}\label{lem8-1}
	Suppose that $w\in H^1(\mathbb{R}^N,\mathbb{C})$ satisfying $\omega_\varepsilon(w,z)=0$ for all $z \in\mathbf{T}_{\Psi_{\sigma_0,\varepsilon}}\mathbf{M}$ with $\sigma_0=(0,0,0,\mu)$ and $\mu\in I$, then there exists a constant $\rho>0$ such that
	\begin{equation}\label{8-1}
	\langle\mathcal{L}_{\eta_{\mu,\varepsilon}}w,w\rangle_\varepsilon\ge \rho||w||_{H^1}^2
	\end{equation}
	where the linearized operator $\mathcal{L}_{\eta_{\mu,\varepsilon}}$ is defined in \eqref{4-1}.
\end{lemma}

Unfortunately, Lemma \ref{lem8-1} can not be applied directly for this multi-soliton case. We need to establish its counterpart. 
\begin{lemma}\label{lem8-2}
	For $u_j$ and $w_j$ defined in \eqref{6-1} \eqref{6-2} \eqref{6-5} \eqref{6-6}, it holds
	\begin{equation}\label{8-2}
	\mathcal{E}_{\mu_j,\varepsilon}(u_j)-\mathcal{E}_{\mu_j,\varepsilon}(\eta_{\mu,\varepsilon})\ge \frac{\rho_j}{2}||w_j||_{H^1}^2-C||w_j||_{H^1}^3+O(e^{-\frac{c}{\varepsilon}}) \ \ \ for \ \ j=1,\cdots,k,
	\end{equation}
	where $\mathcal{E}_{\mu,\varepsilon}$ is defined in \eqref{3-12}, and $\rho_j,c$ are positive constants.
\end{lemma}
\begin{proof}
	We will use Lemma \ref{lem8-1}. Define $\tilde{w}_j=u_j-\eta_{\mu_j,\varepsilon}$, then by \eqref{6-8} and \eqref{6-9} we have
	\begin{equation}\label{8-3}
	||\tilde{w}_j-w_j||_{H^1}=O(e^{-\frac{c}{\varepsilon}}) \ \ \ for \ \ j=1,\cdots,k.
	\end{equation}
	Moreover, since $\omega_\varepsilon(w,\mathcal{T}_jz)=\omega_\varepsilon(\mathcal{T}_j^{-1}w,z)=0, \ \ \forall z \in\mathbf{T}_{\Psi_{\sigma_{0,j},\varepsilon}}\mathbf{M}$ with $\sigma_{0,j}=(0,0,0,\mu_j)$, it holds
	\begin{equation}\label{8-4}
	\omega_\varepsilon(\tilde{w}_j,z)=O(e^{-\frac{c}{\varepsilon}}) \ \ \ \forall z \in\mathbf{T}_{\Psi_{\sigma_{0,j},\varepsilon}}\mathbf{M}.
	\end{equation} 
	Thus we can write $\tilde{w}_j$ as
	\begin{equation}\label{8-5}
	\tilde{w}_j=\hat{w}_j+\sum\limits_{m=1}^{2N+2}r_{j,m}z_{j,m},
	\end{equation}
	where $\hat{w}\in H^1(\mathbb{R}^N,\mathbb{C})$ is the function such that 
	$$\omega_\varepsilon(\hat{w}_j,z)=0, \ \ \forall z \in\mathbf{T}_{\Psi_{\sigma_{0,j},\varepsilon}}\mathbf{M},$$ 
	and $|r_{j,m}|=O(e^{-\frac{c}{\varepsilon}})$ for $m=1,\cdots,2N+2$. By \eqref{8-1} \eqref{8-3} \eqref{8-5}, and using the fact that $\eta_{\mu_j,\varepsilon}$ is a radial ground state solution, we have
	\begin{equation}\label{8-6}
	\begin{split}
	2\mathcal{E}_{\mu_j,\varepsilon}(u_j)-2\mathcal{E}_{\mu_j,\varepsilon}(\eta_{\mu_j,\varepsilon})&=\langle\mathcal{L}_{\eta_{\mu_j,\varepsilon}}\tilde{w}_j,\tilde{w}_j\rangle_\varepsilon-C||\tilde{w}_j||_{H^1}^3\\
	&=\langle\mathcal{L}_{\eta_{\mu_j,\varepsilon}}\hat{w}_j,\hat{w}_j\rangle_\varepsilon-C||\tilde{w}_j||_{H^1}^3+O(e^{-\frac{c}{\varepsilon}})\\
	&\ge \hat{\rho}||\hat{w}_j||_{H^1}^2-C||w_j||_{H^1}^3+O(e^{-\frac{c}{\varepsilon}})\\
	&\ge\rho||w_j||_{H^1}^2-C||w_j||_{H^1}^3+O(e^{-\frac{c}{\varepsilon}}).
	\end{split}
	\end{equation}
	Hence we have proved \eqref{8-2}.
\end{proof}

To study the change of $H^1$ norm for the perturbation term $w$, let us define the norm
\begin{equation}\label{8-7}
\big|||w||\big|_t=\sup_{s\in(0,t]}||w(s)||_{H^1}
\end{equation}
We are going to prove $\big|||w||\big|_t\le c\varepsilon_v$ for $t=\min\{\frac{T_0}{\varepsilon^2},\frac{L}{K\varepsilon}\}$ with $T_0$ some positive constant dependent on $\rho_0:=\min\{\rho_1,\cdots,\rho_k\}$.
\begin{proposition}\label{prop8-3}
	There are positive constants $T_0,c,c'$ independent of $\varepsilon$ and $\varepsilon_v$ such that for $0<t\le \min\{\frac{T_0}{\varepsilon^2},\frac{L}{K\varepsilon}\}$, it holds
	\begin{equation}\label{8-8}
	||w||_{H^1}\le c\varepsilon_v
	\end{equation}
	and
	\begin{equation}\label{8-9}
    |\beta|_\infty\le c'\varepsilon_v^2,
	\end{equation}
	where $\beta$ is defined in \eqref{6-10}-\eqref{6-13} and $|\beta|_\infty=\max|\beta_{j,m}|$.
\end{proposition}
\begin{proof}
	By \eqref{5-18}, we have $||w(0)||_{H^1}\le c\varepsilon_v$, which implies 
	\begin{equation}\label{8-10}
	\sum\limits_{l=1}^k||w_j(0)||_{H^1}\le c\varepsilon_v
	\end{equation}
	from \eqref{6-7}. So for radial ground state solutions $\eta_{\mu_j,\varepsilon}$ where $j=1,\cdots,k$, we have
	\begin{equation}\label{8-11}
	2\sum\limits_{l=1}^k\big(\mathcal{E}_{\mu_l,\varepsilon}(u_l(0))-\mathcal{E}_{\mu_l,\varepsilon}(\eta_{\mu_l,\varepsilon}(0))\big)=2\sum\limits_{l=1}^k\big(||\varepsilon\nabla w_l(0)||_{L^2}^2+\mu_l||w_j(0)||_{L^2}^2-R_{\eta_{\mu_l,\varepsilon}}^{(2)}(w_l(0))\big).
	\end{equation}
	According to \eqref{3-16}, it holds
	\begin{equation}\label{8-12}
	|R_{\eta_{\mu_j,\varepsilon}}^{(2)}(w(0))|\le C||w_j(0)||_{H^1}^2.
	\end{equation}
	So we have 
	\begin{equation}\label{8-13}
	2\sum\limits_{l=1}^k\big(\mathcal{E}_{\mu_l,\varepsilon}(u_l(0))-\mathcal{E}_{\mu_l,\varepsilon}(\eta_{\mu_l,\varepsilon}(0))\big)\le c\varepsilon_v^2
	\end{equation}
	On the other hand, from Lemma \ref{lem8-2} we know that
	\begin{equation}\label{8-14}
	2\sum\limits_{l=1}^k\big(\mathcal{E}_{\mu_l,\varepsilon}(u_l)-\mathcal{E}_{\mu_l,\varepsilon}(\eta_{\mu_l,\varepsilon})\big)\ge\rho_0 \sum\limits_{l=1}^k||w_l||_{H^1}^2-C\sum\limits_{l=1}^k||w_l||_{H^1}^3+O(e^{-\frac{c}{\varepsilon}}),
	\end{equation}
	where $\rho_0=\min\{\rho_1,\cdots,\rho_k\}$. Then we can use Proposition \ref{prop7-2} to integral energy from $0$ to $t$, and deduce that for any $s\in (0,t]$ it holds:
	\begin{equation}\label{8-15}
	\begin{split}
	\rho_0 \sum\limits_{l=1}^k||w_l(s)&||_{H^1}^2-C\sum\limits_{l=1}^k||w_l(s)||_{H^1}^3+O(e^{-\frac{c}{\varepsilon}})\\
	&\le 2\sum\limits_{l=1}^k\big(\mathcal{E}_{\mu_l,\varepsilon}(u_l(0))-\mathcal{E}_{\mu_l,\varepsilon}(\eta_{\mu_l,\varepsilon}(0))\big)+2t\partial_t\big(	\sum\limits_{l=1}^k\mathcal{E}_{\mu_l,\varepsilon}(u_l)-\sum\limits_{l=1}^k\mathcal{E}_{\mu_l}(\eta_{\mu_l,\varepsilon})\big)\\
	&\le c\varepsilon_v^2+c''t\big(|\beta|\cdot\big|||w||\big|_t^2+\varepsilon^2\big|||w||\big|_t^2+\varepsilon\cdot\varepsilon_v^2\big|||w||\big|_t+O(e^{-\frac{c}{\varepsilon}})\big),
	\end{split}
	\end{equation}
	where we have used \eqref{8-13} in the last inequality. By \eqref{6-7} and Cauchy-Schwarz inequality, we have
	$$\sum\limits_{l=1}^k||w_l(s)||_{H^1}^2\ge \frac{1}{k}(\sum\limits_{l=1}^k||w_l(s)||_{H^1})^2\ge\frac{1}{4k}||w(s)||_{H^1}^2$$
	and
	$$\sum\limits_{l=1}^k||w_l(s)||_{H^1}^3\le(\sum\limits_{l=1}^k||w_l(s)||_{H^1})^3\le C||w(s)||_{H^1}^3.$$
	Thus for $t$ satisfying $c''t(|\beta|_\infty+\varepsilon^2)= \rho_0/8k$, we can take the $\sup_{s\in(0,t]}$ of $||w(s)||_{H^1}$ to obtain
	\begin{equation}\label{8-16}
	\frac{\rho_0}{8k} \big|||w||\big|_t^2\le c\varepsilon_v^2+o(1)\cdot\varepsilon_v\big|||w||\big|_t+C\big|||w||\big|_t^3+O(e^{-\frac{c}{\varepsilon}}),
	\end{equation}
	which implies that for $t= \rho_0(8kc''(|\beta|_\infty+\varepsilon^2))^{-1}$, it holds
	\begin{equation}\label{8-17}
	\big|||w||\big|_t^2\le c\varepsilon_v^2
	\end{equation}
	provided $\big|||w||\big|_t\le \rho_0/16kC$ and $\varepsilon$ is sufficiently small. Since $\psi\in U_\delta$, it is always true that $\big|||w||\big|_t\le \delta$ and we can let $\delta\le\rho_0/16kC$. From \eqref{6-42} we have $\beta_{j,m}=X_\delta^{j,m}+O(e^{-\frac{c}{\varepsilon}})$. Combining this and \eqref{8-17}, we can deduce that
	\begin{equation}\label{8-18}
	|\beta|_\infty\le c'\varepsilon_v^2
	\end{equation}
	for $t\le \rho_0(8kc''(\varepsilon_v^2+\varepsilon^2))^{-1}$. According to condition (A) in Section 2 we have $\varepsilon_v=\varepsilon^h$ for some $h>2$, while $\psi\in U_\delta$ requires that $t\le \frac{L}{K\varepsilon}$. Thus \eqref{8-8} and \eqref{8-9} hold over the time interval
	\begin{equation}\label{8-19}
	0<t\le \min\{\frac{T_0}{\varepsilon^2},\frac{L}{K\varepsilon}\},
	\end{equation}
	with $T_0=\rho_0/16kc''<\infty$ a positive constant, which completes the proof of this proposition.
\end{proof}

Now we are ready to prove Theorem \ref{thm1}.

{\bf Proof of Theorem \ref{thm1}:}
From \eqref{6-34}-\eqref{6-37} and Proposition \ref{prop8-3}, for $t\in (0,\min\{\frac{T_0}{\varepsilon^2},\frac{L}{K\varepsilon}\}]$ and $i=1,\cdots,k$ we have
\begin{equation}\label{8-20}
\frac{1}{2}\dot{v}_i=-\varepsilon\nabla V_\varepsilon(a_i)+O(\varepsilon_v^2),
\end{equation}
\begin{equation}\label{8-21}
\frac{\dot{a}_i}{\varepsilon^2}=\frac{v_i}{\varepsilon}+O(\varepsilon_v^2),
\end{equation}
\begin{equation}\label{8-22}
\dot{\gamma}_i=\mu_i-V_\varepsilon(a_i)-\frac{1}{4}v_i^2+O(\varepsilon_v^2),
\end{equation}
\begin{equation}\label{8-23}
\dot{\mu}_i=O(\varepsilon_v^2)
\end{equation}
Note that all the estimates we have done are about the solutions to \eqref{1-3}. However, our aim is to derive the corresponding result for Cauchy problem \eqref{1-1} and \eqref{1-2}. For this purpose, let $t'$ be the time parameter in \eqref{1-1} and \eqref{1-2}, then
\begin{equation}\label{8-24}
t'=\varepsilon t, \ \ \ \ \ \varepsilon\partial_{t'}=\partial_t,
\end{equation}
where $t$ is the time parameter in \eqref{1-3}. So that we can rewrite \eqref{8-20}-\eqref{8-23} to get the asymptotic estimates for Cauchy problem \eqref{1-1} and \eqref{1-2} over the time interval $(0,\min\{\frac{T_0}{\varepsilon},\frac{L}{K}\}]$. In conclusion, we have
\begin{equation}\label{8-25}
||w||_{H^1}=O(\varepsilon_v) 
\end{equation}
and
\begin{equation}\label{8-26}
\frac{1}{2}\dot{v}_i=-\nabla V_\varepsilon(a_i)+O(\frac{\varepsilon_v^2}{\varepsilon}),
\end{equation}
\begin{equation}\label{8-27}
\dot{a}_i=v_i+O(\varepsilon\cdot\varepsilon_v^2),
\end{equation}
\begin{equation}\label{8-28}
\varepsilon\dot{\gamma}_i=\mu_i-V_\varepsilon(a_i)-\frac{1}{4}v_i^2+O(\varepsilon_v^2),
\end{equation}
\begin{equation}\label{8-29}
\varepsilon\dot{\mu}_i=O(\varepsilon_v^2),
\end{equation}
with $i=1,\cdots,k$, and the proof of Theorem \ref{thm1} is therefore complete. \qed \\

Corrollary \ref{coro1} then follows as a consequence of \eqref{5-17} and Theorem \ref{thm1}, which can be used to prove Theorem \ref{thm2}.

{\bf Proof of Theorem \ref{thm2}:}
Let $t_1=\frac{d}{6K}$. Since assumptions (A) (B) and (C) are satisfied with $L=\frac{d}{6}$, by Corollary \ref{coro1}, for $i=1,\cdots,k$ we have $\gamma_i(t_1)\in \mathbb{R}\setminus (0,2\pi]$ and
\begin{equation}\label{8-30}
	v_i(t_1)=v_{i,0}+O(\varepsilon_v),
\end{equation}
\begin{equation}\label{8-31}
	a_i(t_1)=a_{i,0}+t_1v_{i,0}+O(\varepsilon_v),
\end{equation}
\begin{equation}\label{8-32}
	\mu_i(t_1)=	\mu_{i,0}+O(\frac{\varepsilon_v^2}{\varepsilon}).
\end{equation}
Denote $\psi_0^1=\psi(t_1)$ and $\sigma_{i,0}^1=(a_i(t_1),v_i(t_1),\gamma_i(t_1),\mu_i(t_1))$ for $i=1,\cdots,k$. Notice that assumption (D) is satisfied, then by \eqref{8-30} \eqref{8-31} \eqref{8-32}, assumptions (C) is satisfied for $\psi_0^1$ and $\sigma_{i,0}^1$ with $L=\frac{d}{6}$. Using a similar argument, we can extend the maximum time to $t_2=\frac{d}{3K}$ and verify that $\psi_0^2=\psi(t_2)$ and $\sigma_{i,0}^2=(a_i(t_2),v_i(t_2),\gamma_i(t_2),\mu_i(t_2))$ satisfy assumptions (C) with $L=\frac{d}{6}$ when $\varepsilon$ is sufficiently small.
	
Since $\varepsilon_v=\varepsilon^h$ with $h>2$, we can bootstrap to extend the maximum time to $\frac{T_0}{\varepsilon}$, which means the multi-soliton dynamics \eqref{2-3}-\eqref{2-6} can hold for every $t\in(0,T]$ with $T\to +\infty$ as $\varepsilon\to 0$. \qed \\


\phantom{s}
\thispagestyle{empty}


\begin{thebibliography}{99}
	
	\bibitem{Ab}
	W.K. Abou Salem, J. Fr\"{o}hlich and I.M. Sigal, Colliding Solitons for the Nonlinear Schr\"{o}dinger equation,
	\textit{Comm. Math. Phys.}, 291 (2009), 151--176.
	
	\bibitem{A}
	V.I. Arnol'd, Mathematical methods of classical mechanics,
	\textit{Number 60 in Graduate Texts in Mathematics}, Springer-Verlag, New York, Second edition, 1989. 
	
	\bibitem{B1}
	V. Benci, M. Ghimenti and A.M. Micheletti, Nonlinear Schr\"odinger equation: soliton dynamics,
	\textit{J. Differential Equations}, 249 (2010), 3312--3341.
	
	\bibitem{B2}
	V. Benci, M. Ghimenti and A.M Micheletti, On the dynamics of solitons in the nonlinear Schr\"{o}dinger equation,
	\textit{Arch. Rational Mech. Anal.}, 205 (2012), 467--492.
	
	\bibitem{B3}
	V. Benci and D. Fortunato, A minimization method and applications to the study of solitons,
	\textit{Nonlinear Anal.}, 75 (2012), 4398--4421. 
	
	\bibitem{Br}
	J.C. Bronski and R.L. Jerrard, Soliton dynamics in a potential,
	\textit{Math. Res. Lett}, 7 (2000), no. 2--3, 329--342.
	
	\bibitem{CY1}
	D. Cao, E.S. Noussair and S. Yan, Existence and uniqueness results on single-peaked solutions of a semilinear problem,
	\textit{Ann. Inst. H. Poincar\'e Anal. Non Lin\'eaire}, 15 (1998), no. 1, 73--111.
	
	\bibitem{CY2}
	D. Cao, E.S. Noussair and S. Yan, Solutions with multiple peaks for nonlinear elliptic equations,
	\textit{Proc. Roy. Soc. Edinburgh}, 129A (1999), no. 2, 235--264.
	
	\bibitem{CY3}
	D. Cao, E.S. Noussair and S. Yan, Multiscale-bump standing waves with a critical frequency for nonlinear Schr\"{o}dinger equations,
	\textit{Trans. Amer. Math. Soc.}, 360 (2008), no. 7, 3813--3837.
	
	\bibitem{Caz1}
	T. Cazenave and P.L. Lions, Orbital stability of standing waves for some nonlinear Schr\"{o}dinger equations,
	\textit{Comm. Math. Phys.}, 85 (1982), 549--561. 
	
	\bibitem{Caz}
	T. Cazenave, Semilinear Schr\"{o}dinger equations,
	\textit{Courant Lecture Notes in Mathematics, Vol. 10.}, New York University Courant Institute of Mathematical Sciences, New York, 2003.
	
	\bibitem{Fr1}
	J. Fr\"{o}hlich, T.P. Tsai and H.T. Yau, On a classical limit of quantum theory and the non-linear Hartree equation,
	\textit{Geom. Funct. Anal.}, Special Volume, Part \uppercase\expandafter{\romannumeral1}, 2000, 57--78.
	
	\bibitem{Fr2}
	J. Fr\"{o}hlich, T.P. Tsai and H.T. Yau, On the point-particle (Newtonian) limit of the non-linear Hartree equation,
	\textit{Comm. Math. Phys.}, 225 (2002), no. 2, 223--274.
		
	\bibitem{Fr3}
	J. Fr\"{o}hlich, S. Gustafson, B.L.G. Jonsson and I.M. Sigal, Solitary wave dynamics in an external potential,
	\textit{Comm. Math. Phys.}, 250 (2004), no. 3, 613--642.
	
	\bibitem{Fr4}
	J. Fr\"{o}hlich, S. Gustafson, B.L.G. Jonsson and I.M. Sigal, Long time motion of NLS solitary waves in a confining potential.
	\textit{Ann. Henri Poincar\'e}, 7 (2006), no. 4, 621--660.
	
	\bibitem{Gi}
	B. Gidas, W.M. Ni and L. Nirenberg, symmetry and related properties via the maximum principle,
	\textit{Comm. Math. Phys.}, 68 (1979), no. 3, 209--243.
	
	\bibitem{Gri1}
	M. Grillakis, H. Shatah and W. Strauss, Stability theory of solitary waves in the presence of symmetry. \uppercase\expandafter{\romannumeral1}.
	\textit{J. Funct. Anal.}, 74 (1987), no. 1, 160--197.
	
	\bibitem{Gri2}
	M. Grillakis, H. Shatah and W. Strauss, Stability theory of solitary waves in the presence of symmetry. \uppercase\expandafter{\romannumeral2}.
	\textit{J. Funct. Anal.}, 94 (1990), no. 2, 308--348.
	
	\bibitem{Ka}
	T. Kato, Nonlinear Schr\"{o}dinger equations,
	\textit{Lecture Notes in Phys, Vol. 345.}, Springer, Berlin, 1989, 218--263.
	
	\bibitem{Ke1}
	S. Keraani, Semiclassical limit of a class of Schr\"odinger equations with potential,
	\textit{Comm. Partial Differential Equations}, 27 (2002), 693--704.
	
	\bibitem{Ke2}
	S. Keraani, Semiclassical limit of a class of Schr\"odinger equations with potential \uppercase\expandafter{\romannumeral2},
	\textit{Asymptot. Anal.}, 47 (2006), 171--186.
	
	\bibitem{Kw}
	M.K. Kwong, Uniqueness of positive solutions of $\Delta u-u+u^p=0$ in $\mathbb{R}^N$,
	\textit{Arch. Rational Mech. Anal.}, 105 (1989), 243--266.
	
	\bibitem{L1}
	P.L. Lions, The concentration-compactness principle in the calculus of variations. The locally compact case. \uppercase\expandafter{\romannumeral1},
	\textit{Ann. Inst. H. Poincar\'e Anal. Non Lineair\'e}, 1 (1984), no. 2, 109--145.
	
	\bibitem{L2}
	P.L. Lions, The concentration-compactness principle in the calculus of variations. The locally compact case. \uppercase\expandafter{\romannumeral2},
	\textit{Ann. Inst. H. Poincar\'e Anal. Non Lineair\'e}, 1 (1984), no. 4, 223--283.
	
	\bibitem{St}
	M. Struwe, A global compactness result for elliptic boundary value problems involving limiting nonlinearities,
	\textit{Math. Z.}, 187 (1984), no. 4, 511--517.
	
	\bibitem{WY}
	J. Wei and S. Yan, Infinitely many positive solutions for nonlinear Schr\"{o}dinger equations in $\mathbb{R}^N$,
	\textit{Calc. Var. Partial Differential Equations}, 37 (2010), no. 3--4, 423--439.
	
	\bibitem{We1}
	M. Weinstein, Modulational stability of ground states of nonlinear Schr\"{o}dinger equations,
	\textit{SIAM J. Math. Anal.}, 16 (1985), 472--491.
	
	\bibitem{We2}
	M. Weinstein, Lyapunov stability of ground states of nonlinear dispersive evolution equations,
	\textit{Comm. Pure Appl. Math.}, 39 (1986), no. 1, 51--67.
	
	\bibitem{W}
	M. Willem, Minimax theorems,
	\textit{Progress in Nonlinear Differential Equations and their Applications}, 24. Birkh\"{a}user Boston, Inc., Boston, MA, 1996.
	
	
	
\end{thebibliography}
\end{document}